\documentclass[twoside,11pt]{article}

\usepackage{
	adjustbox,
	algorithm,
	algpseudocodex,
	amsmath,
	blkarray,
	bm,
	mathtools,
	mathrsfs,
	pgfplotstable,
	soul,
	subcaption,
	tikz,
	xcolor
}

\usepackage[preprint]{jmlr2e}

\newcommand{\R}{\mathbb R}

\newcommand{\E}{\mathbb E}
\renewcommand{\P}{\mathbb P}

\newcommand{\bbH}{\mathbb H}
\newcommand{\F}{\mathscr F}
\renewcommand{\L}{\mathcal L}

\newcommand{\bA}{\mathbf A}
\newcommand{\bB}{\mathbf B}
\newcommand{\bC}{\mathbf C}

\newcommand{\bG}{\mathbf G}

\newcommand{\bI}{\mathbf I}

\newcommand{\bP}{\mathbf P}

\newcommand{\bS}{\mathbf S}
\newcommand{\bT}{\mathbf T}
\newcommand{\bU}{\mathbf U}
\newcommand{\bV}{\mathbf V}

\newcommand{\bX}{\mathbf X}
\newcommand{\bY}{\mathbf Y}
\newcommand{\bZ}{\mathbf Z}

\newcommand{\bOmega}{\mathbf\Omega}

\newcommand{\bSigma}{\mathbf\Sigma}
\newcommand{\bTheta}{\mathbf\Theta}
\newcommand{\bmu}{\boldsymbol{\mu}}

\newcommand{\cC}{\mathcal C}
\newcommand{\cD}{\mathcal D}

\newcommand{\cH}{\mathcal H}
\newcommand{\cK}{\mathcal K}
\newcommand{\cM}{\mathcal M}
\newcommand{\cN}{\mathcal N}

\newcommand{\cR}{\mathcal R}

\newcommand{\cU}{\mathcal U}

\newcommand{\cW}{\mathcal W}
\newcommand{\cX}{\mathcal X}
\newcommand{\cY}{\mathcal Y}

\newcommand{\sH}{\mathscr H}

\newcommand{\dd}{\,\mathrm d}

\newcommand{\as}{\mathrm{a.s.}}
\newcommand{\Fr}{\mathrm{F}}
\newcommand{\GP}{\mathcal{GP}}

\newcommand{\HS}{\mathrm{HS}}
\newcommand{\red}{\mathrm{red}}
\newcommand{\green}{\mathrm{green}}
\newcommand{\tree}{\mathrm{tree}}
\newcommand{\future}{\mathrm{future}}

\newcommand{\nnn}[1]{{\vert\kern-0.25ex\vert\kern-0.25ex\vert #1 
    \vert\kern-0.25ex\vert\kern-0.25ex\vert}}
\newcommand{\snnn}[1]{{\left\vert\kern-0.25ex\left\vert\kern-0.25ex\left\vert #1 
    \right\vert\kern-0.25ex\right\vert\kern-0.25ex\right\vert}}

\DeclarePairedDelimiter\brak{\langle}{\rangle}
\DeclarePairedDelimiter\ceil{\lceil}{\rceil}

\DeclareMathOperator{\col}{col}
\DeclareMathOperator{\diag}{diag}

\DeclareMathOperator{\dist}{dist}

\DeclareMathOperator{\rank}{rank}
\DeclareMathOperator{\polylog}{polylog}
\DeclareMathOperator{\Tr}{Tr}
\DeclareMathOperator{\vol}{vol}

\definecolor{green+}{HTML}{3eff3e}
\definecolor{red+}{HTML}{c70227}

\definecolor{rev}{HTML}{9c89ff}
\definecolor{rev1}{HTML}{cb270f}
\definecolor{rev2}{HTML}{1c8235}

\algrenewcommand\algorithmicrequire{\textbf{Input:}}
\algrenewcommand\algorithmicensure{\textbf{Output:}}

\usetikzlibrary{arrows.meta, matrix}
\pgfplotsset{compat=1.18}

\usepackage{hyperref}

\usepackage{lastpage}
\jmlrheading{25}{2025}{1-\pageref{LastPage}}{12/23; Revised 7/24}{}{21-0000}{Christopher Wang and Alex Townsend}

\ShortHeadings{Operator Learning for Hyperbolic PDEs}{Wang and Townsend}
\firstpageno{1}

\begin{document}

\title{Operator learning for hyperbolic partial differential equations}

\author{\name Christopher Wang \email cyw33@cornell.edu \\
       \addr Department of Mathematics\\
       Cornell University \\
       Ithaca, NY 14853, USA
       \AND
       \name Alex Townsend \email townsend@cornell.edu \\
       \addr Department of Mathematics\\
       Cornell University \\
       Ithaca, NY 14853, USA}

\editor{Dr.~Fei Sha}

\maketitle

\begin{abstract}%   <- trailing '%' for backward compatibility of .sty file
We construct the first rigorously justified probabilistic algorithm for recovering the solution operator of a hyperbolic partial differential equation (PDE) in two variables from input-output training pairs. The primary challenge of recovering the solution operator of hyperbolic PDEs is the presence of characteristics, along which the associated Green's function is discontinuous. Therefore, a central component of our algorithm is a rank detection scheme that identifies the approximate location of the characteristics. By combining the randomized singular value decomposition with an adaptive hierarchical partition of the domain, we construct an approximant to the solution operator using $O(\Psi_\epsilon^{-1}\epsilon^{-7}\log(\Xi_\epsilon^{-1}\epsilon^{-1}))$ input-output pairs with relative error $O(\Xi_\epsilon^{-1}\epsilon)$ in the operator norm as $\epsilon\to0$, with high probability. Here, $\Psi_\epsilon$ represents the existence of degenerate singular values of the solution operator, and $\Xi_\epsilon$ measures the quality of the training data. Our assumptions on the regularity of the coefficients of the hyperbolic PDE are relatively weak given that hyperbolic PDEs do not have the ``instantaneous smoothing effect'' of elliptic and parabolic PDEs, and our recovery rate improves as the regularity of the coefficients increases.
\end{abstract}


\section{Introduction}
    
    In this paper, we consider the recovery of the solution operator associated with hyperbolic partial differential equations (PDEs) from data. Generally, the task of PDE learning is to capture information about an unknown, inhomogeneous PDE given data corresponding to the observed effect of the PDE on input functions \citep{Fan2019, Fan2020, Gin2021, Karniadakis2021, Kovachki2023, Li2020a, Li2020b, Li2021, Lu2021a, Lu2021b, WangS2021}. These PDEs typically represent real-world dynamical systems whose governing principles are poorly understood, even though one can accurately observe or predict their evolutions, either through experimentation or through simulation. Data-driven PDE learning has applications in climate science \citep{Bi2023, Lam2023}, biology \citep{Raissi2020}, and physics \citep{Chen2021, Kochkov2021, Kutz2017, Qian2020}, and is a significant area of research in scientific machine learning and reduced order modeling \citep{Berman2023, Berman2024, Brunton2016, Chen2023, deHoop2023, Krishnapriyan2021, Rudy2017, Subramanian2023, Zhang2018}. In particular, many effective practical schemes based on neural networks have been developed to recover the solution operator associated with an unknown PDE \citep{Boulle2022d, Feliu2020, Fan2020, Gin2021, Li2020a, Li2020b, Li2021, Wan2023, WangS2021}, although the inscrutability of these neural networks as ``black boxes'' often prevents one from understanding the underlying explanation for their success. Moreover, theoretical research in PDE learning typically centers on error estimates for neural network-based schemes \citep{Kovachki2023, Lanthaler2022, Lu2021a}. There is a growing body of research for operator learning in the context of elliptic and parabolic PDEs \citep{Boulle2022b, Boulle2022a, Schafer2023, Schafer2021}, but there is a lack of theoretical work on solution operators of hyperbolic PDEs. The existing work focuses primarily on deep learning models for solving hyperbolic PDEs \citep{Arora2023, Berman2024, Bruna2024, Guo2020, Huang2023, RodriguezTorrado2022, Thodi2023} or related operators associated with inverse problems \citep{Khoo2019, Li2022, Molinaro2023}, rather than recovering their solution operators. Therefore, we aim to shed light on the theoretical aspects of solution operator learning for hyperbolic PDEs and to open the field for further research.
    
    We consider an unknown second-order hyperbolic linear partial differential operator (PDO) in two variables of the form:
    \begin{subequations}\label{eq:problem}
    \begin{equation}\label{eq:L}
    	\L u := u_{tt} - a(x,t)u_{xx} + b(x,t)u_x + c(x,t)u, \qquad (x,t)\in D_T:=[0,1]\times[0,1]
	\end{equation}
	together with homogeneous initial-boundary conditions
	\begin{equation}\label{eq:homogeneous}
		\left\{\begin{alignedat}{2}
			u(x,0) = u_t(x,0) &= 0, && \qquad 0 \le x \le 1 \\
			u(0,t) = u_x(0,t) &= 0, && \qquad 0 \le t \le 1 \\
			u(1,t) = u_x(1,t) &= 0, && \qquad 0 \le t \le 1 \\
		\end{alignedat}\right\}.
	\end{equation}
	\end{subequations}
	We assume that the coefficients are somewhat regular, namely, $a,b,c\in \cC^1(D_T)$, and that $\L$ is strictly hyperbolic, i.e., $a>0$, and self-adjoint, i.e., $a_x+b=0$.\footnote{We discuss relaxations of these constraints, as well as inhomogeneous initial conditions, in Section \ref{s:conclusion}.} For equations of the form $\L u = f$, the function $f$ is called the forcing term of the PDE, while $u$ is the corresponding system's response or solution. The hyperbolic equation $\L u = f$ describes wave-like phenomena, especially in heterogeneous media, such as water waves, acoustic and electromagnetic signals, or earthquakes \citep{Evans2010, Lax2006}.
	
	Our learning task is to approximate the solution operator that maps forcing terms to responses, given training data in the form of input-output pairs $\{(f_j,u_j)\}_{j=1}^N$ that satisfy either the Cauchy problem \eqref{eq:problem} or the corresponding adjoint Cauchy problem (see Section 4 for details). Associated with the Cauchy problem \eqref{eq:problem} is a unique Green's function $G:D_T\times D_T\to\R$, which is a kernel for the solution operator \citep{Courant1962, Mackie1965}. That is, solutions to the problem \eqref{eq:problem} are given by the integral operator
	\begin{equation}\label{eq:G}
		u(x,t) = \int_{D_T} G(x,t;y,s) f(y,s) \dd y\dd s, \qquad (x,t)\in D_T,
	\end{equation}
	for $f\in L^2(D_T)$. We call $G$ the {\em homogeneous Green's function}. Our goal is to recover the action of the integral operator in \eqref{eq:G} as accurately as possible, measured by the operator norm. Notably, {\em we are not solving an inverse problem}, in that we are not interested in learning the coefficients of \eqref{eq:L}.\footnote{Learning the coefficients of an equation can be done with techniques similar to sparse identification of nonlinear dynamics (SINDy) \citep{Brunton2016, Kaiser2018}.} In fact, due to the compactness of the solution operator, its recovery is well-posed, and our proposed algorithm is stable (see Section \ref{ss:HS-op}). Rather, our task is to construct a direct solver for the Cauchy problem \eqref{eq:problem} without explicit knowledge of the equation's coefficients.
		
	\subsection{Challenges and contributions}
	
	This paper describes the first theoretically justified scheme for recovering the solution operator associated with hyperbolic PDEs using only input-output data. We also provide a rigorous rate of recovery for our algorithm. While we adopt some of the strategies used by \citet{Boulle2022b} and \citet{Boulle2022a}, such as the randomized singular value decomposition (rSVD), we face three unique circumstances when recovering the Green's functions of hyperbolic PDEs that make our situation challenging:
	
	\subsubsection*{Characteristic curves}
	The primary difficulty of recovering the Green's function of a hyperbolic PDE is the presence of characteristic curves (or simply characteristics), along which the Green's function is highly irregular. Characteristics describe the trajectory of waves in spacetime and are determined by the coefficients of \eqref{eq:L}, so their location is also unknown to us in advance. Thus, our recovery algorithm needs to detect if a region of the Green's function's domain intersects a characteristic curve, using only input-output data.
	
	\subsubsection*{Adaptive partitioning}
	For elliptic and parabolic PDEs, the Green's functions are numerically low-rank off the diagonal, so they are well-suited to an approximation strategy via hierarchical partitioning of the domain \citep{Bebendorf2003, Boulle2022b}. Since the Green's functions of hyperbolic PDEs have numerically high rank not only on the diagonal but also along the characteristics, we cannot apply a naive hierarchical partition of the domain. Instead, we use an adaptive partitioning strategy, which at each level uses information from the rSVD to decide which regions to partition further.
		
	\subsubsection*{Regularity of coefficients}
	Underlying the difficulty posed by characteristic curves is a more fundamental issue with hyperbolic PDEs: irregularities are not instantaneously smoothed and propagate. This feature of Green's functions makes them challenging to recover, as the variable coefficients can create additional discontinuities. Accordingly, the recovery rate we derive for the hyperbolic case depends on the regularity of the coefficients; the more regular the coefficients, the faster the recovery.

	\vspace{2mm}
	To overcome these challenges, we rely on the following three properties of the Green's function $G$ of the hyperbolic PDO of \eqref{eq:L} (see Section \ref{s:Green}):
	\begin{enumerate}
		\item The function $G$ is square-integrable, with jump discontinuities along the characteristics and on the diagonal of $D_T\times D_T$.
		\item Away from the diagonal and the characteristics, $G$ is regular and, consequently, numerically low-rank.
		\item The characteristics form a piecewise regular hypersurface in $D_T\times D_T$ and never ``accumulate'' in finite time. In other words, the volume of a tube of radius $\delta$ around the characteristic surface shrinks to zero as $\delta\to0$.
	\end{enumerate}
	
	Our main technical contribution is the derivation of a rigorous probabilistic algorithm that constructs an approximant to the solution operator associated with \eqref{eq:L} using randomly generated input-output data $\{(f_j,u_j)\}_{j=1}^N$, with a small error in the operator norm. In particular, we show in Theorem \ref{thm:main} that with high probability, the solution operator can be recovered within a tolerance of $\Xi_\epsilon^{-1}\epsilon$ using $O(\Psi_\epsilon^{-1}\epsilon^{-7} \log(\Xi_\epsilon^{-1}\epsilon^{-1}))$ input-output pairs, where $\Xi_\epsilon$ and $\Psi_\epsilon$ are defined in \eqref{eq:Xi} and \eqref{eq:Psi} and describe features of the operator $\L$ and the input-output data---namely, the size of the singular value gaps of the solution operator, and the quality of the covariance kernel used to generate the training data---that cannot be controlled without additional assumptions. Our construction relies on an adaptive hierarchical partition of the spatio-temporal domain, which roughly identifies the location of the characteristic curves, combined with the rSVD for HS operators. We remark that the learning rate is comparable to the one derived by \citet{Boulle2022a} for elliptic PDEs, where it was later shown that an $\epsilon^{-6}$ factor can be improved to $\polylog(1/\epsilon)$ using the peeling algorithm; see Section \ref{ss:peeling} \citep{Boulle2023, Levitt2022, Lin2011}. We also emphasize that our scheme succeeds for equations with \emph{both space- and time-dependent} coefficients, which informs the use of the time domain rather than the frequency domain in our analysis. While we are interested in as general of a setting as possible, other schemes that exploit the Helmholtz equation through the frequency domain may be preferable when the coefficients are time-independent \citep{Anderson2020, Liu2023, ZepedaNunez2016}.
	
	A key component of our probabilistic algorithm is a scheme that detects the numerical rank of an operator using input-output data, which allows us to tell whether or not a domain intersects a characteristic. We do so by showing, in Theorem \ref{thm:rSVA}, that the rSVD can efficiently recover an operator's dominant singular subspaces. We then use the singular subspaces to approximate the dominant singular values of the operator, whose decay rate corresponds to the operator's numerical rank. While our scheme assumes the existence of a gap between adjacent singular values, such an assumption is reasonable in practice since the singular values of the solution operator typically exhibit decay, which is fast enough for efficient numerical rank detection \citep{Meier2024}. Still, we believe that the assumption of a singular value gap can be discarded, although this will not be discussed in the present work.
	
	Our rank detection scheme facilitates the adaptive feature of the partitioning strategy, since at each hierarchical level we partition only the subdomains flagged as numerically high-rank. While such adaptive strategies are not new and have been previously applied to wave-like settings \citep{Liu2021, Massei2022, ZepedaNunez2016}, our work supplies, to our knowledge, the first theoretical guarantee that the rSVD can be used in an adaptive scheme in a stable manner and with high probability of success. The use of the rSVD is particularly important in the realm of operator learning, where one often only has access to input-output data and thus cannot compute a SVD directly.
	
	Finally, in Theorem \ref{thm:rSVD-op}, we improve the error estimates for the rSVD for HS operators, as derived by \citet{Boulle2022a}, by a factor of $\sqrt k$, where $k$ is the target rank of the constructed approximant. Assuming one always oversamples using an additional $k$ training pairs, then the error factor in \citet[Thm.~1]{Boulle2022a} grows to infinity roughly like $O(k)$ as $k\to\infty$, whereas, practically speaking, our error factor remains bounded. This improvement shows that the rSVD for HS operators behaves similarly to the rSVD for matrices; our bounds are asymptotically comparable to the bounds proved in \citet{Halko2011}.
	
	We remark that because our algorithm is mainly of theoretical value, we choose to work in the continuous setting---that is, without considering discretization---both to simplify the analysis and to maintain discretization-invariance of our theoretical guarantees.  It may be desirable to learn the Green's function as an operator between infinite-dimensional function spaces rather than to learn a space- and time-discretized version of the Green's function. Indeed, \citet{Huang2024} shows that learning the solution operator as a ``function-to-function'' map and only discretizing its inputs and outputs when necessary may be more data-efficient than learning its discretization as a ``vector-to-vector'' map. In practice, learning a solution operator without discretizing in space or time can be achieved by working in a finite-dimensional subspace of $L^2$ using, for instance, the Legendre basis to represent functions, as done in the MATLAB package {\tt chebfun} \citep{chebfun}. Nevertheless, we implement a space- and time-discretized version of our algorithm in Section \ref{s:numerics}, demonstrating its robustness to discretization.

	\subsection{Organization}
	
	The paper proceeds as follows. In Section \ref{s:Green}, we review the characteristics of hyperbolic PDEs and their relationship with the Green's function, while Section \ref{s:rSVD} develops the necessary tools from randomized linear algebra, namely, the rSVD for Hilbert--Schmidt (HS) operators. These tools are employed in Section \ref{s:main} to construct our probabilistic algorithm for recovering the solution of a hyperbolic PDO using input-output training pairs; we also analyze the recovery rate and probability of success. A numerical implementation and example of our algorithm is presented in Section \ref{s:numerics}. Finally, we summarize our results and discuss further directions of research in Section \ref{s:conclusion}. Background material on HS operators, quasimatrices, orthogonal projectors, Gaussian processes (GPs), and the Legendre basis can be found in \hyperref[appendix-A]{Appendix A}. Proofs related to the rSVD appear in \hyperref[appendix-B]{Appendix B}.

	\subsection{Notation}
	
	Throughout the paper, we use the following notation. Norms are denoted by $\|\cdot\|$, and the type of norm is given by a subscript. If the argument is an operator, then $\|\cdot\|$ without a subscript denotes the operator norm. For a matrix or quasimatrix $\bA$ (see Section \ref{ss:quasimatrices}), we denote its Moore--Penrose pseudoinverse by $\bA^\dagger$. We write $\bI_m$ to denote an $m\times m$ identity matrix. For a random object $\bX$, we write $\bX\sim\cD$ to mean that $\bX$ is drawn from the distribution $\cD$; we write $\bX\sim\bY$ to mean that $\bX$ has the same distribution as a different random object $\bY$. We use $\cN(\mu,\sigma^2)$ to denote the univariate Gaussian distribution with mean $\mu$ and variance $\sigma^2$. For a vector $\bmu\in\R^n$ and a symmetric positive definite $n\times n$ matrix $\bC$, we write $\cN(\bmu,\bC)$ for the multivariate Gaussian distribution with mean $\bmu$ and covariance matrix $\bC$. For an integer $r\ge0$, $\cC^r(D)$ denotes the space of $r$-times continuously differentiable functions on a domain $D$; if $D$ is closed, then the functions in $\cC^r(D)$ are also required to extend continuously to the boundary of $D$. The other function space we use is the Hilbert space $L^2(D)$ of square-integrable functions on $D$.

	\section{Green's functions of hyperbolic PDOs}\label{s:Green}
	
	Our recovery scheme relies crucially on understanding where the singularities of the Green's function lie. In this section, we investigate the geometry of characteristic curves and discuss the dependence of the regularity of the Green's function on the regularity of the coefficients of $\L$ as in \eqref{eq:L}.  We assume that the coefficients of $\L$ have regularity $a,b,c\in \cC^r(D_T)$, for integer $r\ge1$.

	\subsection{Green's functions in the domain \texorpdfstring{$\bbH$}{H}}
	
	We first consider the homogeneous Green's function $G^\bbH$ of $\L$, as defined in \eqref{eq:G}, in the unbounded domain $\bbH = \R\times[0,\infty)$.\footnote{Much of the classical literature discusses what is known as the {\em Riemann function}, rather than the Green's function, of hyperbolic equations. The Riemann function (also referred to as the Riemann--Green function or radiation solution) is a fundamental solution of the PDE and allows one to write an integral solution to the homogeneous equation given Cauchy initial data. Essentially, every result for the Riemann function also holds for the homogeneous Green's function, as they are closely related \citep{Mackie1965}.} Since we assume the coefficients $a,b,c$ are at least continuously differentiable, $G$ exists and is unique \citep[Ch.~V.5--6]{Courant1962}. Its discontinuities lie on the characteristics, which we now describe.
	
	The characteristic curves passing through some point $(x_0,t_0)\in\bbH$ can be interpreted as the paths in spacetime traversed by waves propagating from an instantaneous unit force at $(x_0,t_0)$. We obtain a family of characteristic curves by iterating over all $(x_0,t_0)\in\bbH$. For hyperbolic equations in two variables, the characteristic curves are the graphs of solutions to the following ordinary differential equations:
	\begin{equation}\label{eq:char-ODE}
		x'(t) + \sqrt{a(x(t),t)} = 0, \qquad x'(t) - \sqrt{a(x(t),t)} = 0
	\end{equation}
	over all choices of initial conditions \citep[Ch.~III.1]{Courant1962}. Thus, for any $(x_0,t_0)\in\bbH$, there exist exactly two characteristic curves passing through $(x_0,t_0)$, given by the equations in \eqref{eq:char-ODE}, both of which are of class $\cC^{r+1}$ and satisfy the initial condition $x(t_0)=x_0$. Viewed as functions of $t$, one is strictly increasing, and the other is strictly decreasing, so they intersect transversally at $(x_0,t_0)$ and partition $\bbH$ into four connected components lying respectively to the north, south, east, and west of $(x_0,t_0)$. We refer to the component to the north---that is, forward in time---as the {\em future light cone} $\Lambda_\future(x_0,t_0)$. Additionally, we call the characteristic ``ray'' emanating toward the northwest of $(x_0,t_0)$ the {\em negative characteristic ray}, and likewise we call the ray emanating toward the northeast the {\em positive characteristic ray} (see Figure \ref{fig:char-H}).

	\begin{figure}[t]
		\centering
		\begin{tikzpicture}[domain=0:7,
			dot/.style={circle, fill=black, inner sep=0pt, minimum size=5pt, node contents={}}]
        
  			\draw[latex-latex] (-1,0) -- (7,0) node[right] {$x$};			% x-axis
			\draw[-latex] (0,0) -- (0,6) node[above] {$t$};					% t-axis
			\draw[-Latex, scale=2.5, domain=1:1.76, smooth, variable=\y, line width=0.4mm]
				plot ({sinh(-1.385 + 2*sqrt(\y+1))-1}, \y)
				node[above] {$\zeta_+(x_0,t_0)$};							% pos char, future
			\draw[-Latex, scale=2.5, domain=1:2.132, smooth, variable=\y, line width=0.4mm]
				plot ({-sinh(-4.272 + 2*sqrt(\y+1))-1}, \y)
				node[left] {$\zeta_-(x_0,t_0)$};							% neg char, future
			\draw[dashed, scale=2.5, domain=0.121:1, smooth, variable=\y, line width=0.4mm]
				plot ({sinh(-1.385 + 2*sqrt(\y+1))-1}, \y);					% pos char, past
			\draw[dashed, scale=2.5, domain=0.362:1, smooth, variable=\y, line width=0.4mm]
				plot ({-sinh(-4.272 + 2*sqrt(\y+1))-1}, \y);				% neg char, past
			\draw (2.5,2.5) node[dot, label={[label distance=2mm]below:{$(x_0,t_0)$}}]; % initial
			
			\draw (3,4.5) node {$\Lambda_\future(x_0,t_0)$};	% future light cone
		%	\draw (3,0.8) node {$\Lambda_\past(x_0,t_0)$};		% past light cone
		\end{tikzpicture}
		\caption{Characteristics associated with the coefficient $a(x,t) = \frac{(x+1)^2+1}{t+1}$ in the domain $\bbH$, with initial point $(x_0,t_0)=(1,1)$. They represent wave trajectories in spacetime produced by a unit force at $(1,1)$: solid curves are future trajectories, while dashed curves are past trajectories. Positive and negative characteristic rays emanating from $(x_0,t_0)$ are solid and labeled by $\zeta_\pm(x_0,t_0)$. The future light cone is the region labeled $\Lambda_\future$.} 
		\label{fig:char-H}
	\end{figure}
	
	We are interested in the ``bundle'' of positive and negative characteristic rays indexed over all initial points $(x_0,t_0)\in\bbH$, since they determine where the Green's function is irregular. Let $\zeta_\pm(x_0,t_0)$ denote the positive and negative characteristic rays, respectively, emanating from $(x_0,t_0)$. Then we define
	\begin{equation}\label{eq:Z-pm-H}
		Z_\pm^\bbH := \{(x,t,x_0,t_0)\in\bbH\times\bbH : (x,t) \in \zeta_\pm(x_0,t_0)\}
	\end{equation}
	as well as
	\begin{equation}\label{eq:Z-H}
		Z^\bbH := Z_+^\bbH \cup Z_-^\bbH.
	\end{equation}
	Observe that $Z_\pm^\bbH$ are 3-dimensional $\cC^r$-manifolds with boundary in $\bbH\times\bbH$.\footnote{Here, a $\cC^r$-manifold is simply a manifold whose transition maps are of class $\cC^r$. The claim here follows from the observation that $Z_\pm^\bbH$ can be defined as the graph of the flow generated by the time-dependent vector field $(x,t)\mapsto\sqrt{a(x,t)}$. The details are irrelevant here, and we omit them.}
	
	For fixed $(x_0,t_0)\in D_T$, the slices $G^\bbH_{x_0,t_0}$ given by $(x,t)\mapsto G^\bbH(x,t;x_0,t_0)$ are $r$-times continuously differentiable as long as $(x,t)$ does not lie on either of the characteristic rays emanating from $(x_0,t_0)$ \citep[Thm.~1]{Lerner1991}. Moreover, due to symmetries of fundamental solutions of hyperbolic equations \citep[Ch.~V.5]{Courant1962}, we conclude that $G^\bbH$ has the same regularity in every variable.
	
	For hyperbolic equations in two variables, the singularity of the Green's function on the characteristics is a jump discontinuity. This is because $G^\bbH_{x_0,t_0}$ restricted to $\Lambda_\future(x_0,t_0)$ satisfies continuous boundary conditions on the characteristic rays \citep[Ch.~V.5]{Courant1962}. Outside the future light cone, the Green's function is identically zero \citep{Mackie1965}. In summary, we have 
	\begin{equation}\label{eq:G-reg-H}
		G^\bbH\in \cC^r((\bbH\times\bbH)\setminus Z^\bbH),
	\end{equation}
	where $Z^\bbH$ is defined by \eqref{eq:Z-H}.

	\subsection{Green's functions in the domain \texorpdfstring{$D_T$}{D\_T}}\label{ss:char-boundary}
	
	\begin{figure}[t]
		\centering
		\begin{tikzpicture}[domain=-1:7,
			dot/.style={circle, fill=black, inner sep=0pt, minimum size=5pt, node contents={}}]
			
  			\draw[latex-latex] (-1,0) -- (7,0) node[right] {$x$};	% x-axis
			\draw[-latex] (0,0) -- (0,9.5) node[above] {$t$};		% t-axis (t=0)
			\draw[-latex] (6,0) -- (6,9.5);							% t-axis (t=2)
			\draw[-Latex, scale=3, domain=1:1.566, smooth, variable=\y, line width=0.4mm] 
				plot ({sinh(-1.385 + 2*sqrt(\y+1))-1}, \y)
				node[left] {$\zeta_+^{(0)}(x_0,t_0)$\quad \ }
				node[right] {$t_+^{(0)}$};							% pos char, future
			\draw[dashed, scale=3, domain=0.284:1, smooth, variable=\y, line width=0.4mm]
				plot ({sinh(-1.385 + 2*sqrt(\y+1))-1}, \y);			% pos char, past
			\draw[dashed, scale=3, domain=0:0.284, smooth, variable=\y, line width=0.4mm] 
				plot ({-sinh(-3.148 + 2*sqrt(\y+1))-1}, \y);		% pos char reflect, past
			\draw[-Latex, scale=3, domain=1.566:3, smooth, variable=\y, line width=0.4mm]
				plot ({-sinh(-5.022 + 2*sqrt(\y+1))-1}, \y)
				node[above]	{\qquad\ \ $\zeta_+^{(1)}(x_0,t_0)$};	% pos char reflect, future
			\draw[-Latex, scale=3, domain=1:1.874, smooth, variable=\y, line width=0.4mm]
				plot ({-sinh(-4.272 + 2*sqrt(\y+1))-1}, \y)
				node[right] {\quad $\zeta_-^{(0)}(x_0,t_0)$}
				node[left] {$t_-^{(0)}$};							% neg char, future
			\draw[dashed, scale=3, domain=0.505:1, smooth, variable=\y, line width=0.4mm]
				plot ({-sinh(-4.272 + 2*sqrt(\y+1))-1}, \y);		% neg char, past
			\draw[-Latex, scale=3, domain=1.874:3, smooth, variable=\y, line width=0.4mm]
				plot ({sinh(-2.509 + 2*sqrt(\y+1))-1}, \y)
				node[above right] {$\zeta_-^{(1)}(x_0,t_0)$};		% neg char reflect, future
			\draw[dashed, scale=3, domain=0:0.505, smooth, variable=\y, line width=0.4mm]
				plot ({sinh(-0.635 + 2*sqrt(\y+1))-1}, \y);			% neg char reflect, past
			\draw (3,3) node[dot, label={[label distance=2mm]below:{$(x_0,t_0)$}}]; % initial
		\end{tikzpicture}
		\caption{In a bounded domain, the characteristics reflect off the boundary, producing a series of ``reflecting characteristic segments,'' depicted by solid curves labeled $\zeta_\pm^{(j)}(x_0,t_0)$, for $j=0,1,2,\dots$. ``Collision points'' are labeled $t_\pm^{(j)}$.}
		\label{fig:char-boundary}
	\end{figure}
	
	When we consider $\L$ on the bounded domain $D_T = [0,1]\times[0,1]$, we must also establish homogeneous boundary conditions on $\{0,1\}\times[0,1]$, in addition to homogeneous initial conditions. Homogeneous boundary conditions produce wave reflections, so the characteristic curves ``reflect'' off the boundaries. When a positive characteristic ray collides with the right boundary, its reflection is given by the negative characteristic ray emanating from the collision point. Likewise, when a negative characteristic ray collides with the left boundary, its reflection is given by the positive characteristic ray emanating from the collision point.
	
	Wave trajectories can thus be described as a series $\zeta_\pm^{(0)}(x_0,t_0),\zeta_\pm^{(1)}(x_0,t_0),\dots$ of ``reflecting characteristic segments'' emanating from $(x_0,t_0)$, which terminates at some index $N(x_0,t_0)$ once the segments cross the time horizon $t=1$ (see Figure \ref{fig:char-boundary}). Letting $N = \max_{(x_0,t_0)\in D_T}N(x_0,t_0)$, we define, analogous with \eqref{eq:Z-pm-H} and \eqref{eq:Z-H}, the objects
	\begin{equation}\label{eq:Z-pm-boundary}
		Z_\pm^{(j)} := \{(x,t,x_0,t_0)\in D_T\times D_T : (x,t) \in \zeta_\pm^{(j)}(x_0,t_0)\},
			\qquad 0\le j\le N,
	\end{equation}
	as well as
	\begin{equation}\label{eq:Z-boundary}
		Z := \bigcup_{j=0}^N \left(Z_+^{(j)} \cup Z_-^{(j)}\right).
	\end{equation}
	Again, each component $Z_\pm^{(j)}$, $j=1,\dots,N$ is a 3-dimensional $\cC^r$-manifold with boundary in $D_T\times D_T$, so that $Z$ is a piecewise $\cC^r$-manifold. Notice that the number of reflections $N$ is bounded by $\max_{(x,t)\in D_T}\sqrt{a(x,t)}$.\footnote{The claims made in this section can be seen by applying the method of reflections to solve the homogeneous initial-boundary problem which defines the homogeneous Green's function in a bounded domain \citep[see, e.g.,][]{Laurent2021}. Again, the details are not relevant to us.}
	
	Besides the reflecting characteristics, the homogeneous Green's function $G$ in the domain $D_T$ has the same regularity properties as the homogeneous Green's function $G^\bbH$ in the domain $\bbH$. The reflections produced by homogeneous boundary conditions manifest as additional jump discontinuities for $G$, lying on the reflecting characteristic segments described in Section \ref{ss:char-boundary}. In other words, on a bounded domain, we have
	\begin{equation}\label{eq:G-reg-boundary}
		G\in \cC^r((D_T\times D_T)\setminus Z),
	\end{equation}
	where $Z$ is defined by \eqref{eq:Z-boundary}.
	
	\begin{figure}[h]
		\centering
		\begin{tikzpicture}[domain=0:7,
			dot/.style={circle, fill=black, inner sep=0pt, minimum size=5pt, node contents={}}]
        
  			\draw[latex-latex] (-1,0) -- (7,0) node[right] {$x$};	% x-axis
			\draw[-latex] (0,0) -- (0,6) node[above] {$t$};			% t-axis (t=0)
			\draw[-latex] (6,0) -- (6,6);							% t-axis (t=2)
			\draw[-, domain=1.5:6, smooth, variable=\x, line width=0.4mm]
				plot (\x,0.5+\x/3);
			\draw[-, domain=0:6, smooth, variable=\x, line width=0.4mm]
				plot (\x,4.5-\x/3);
			\draw[-Latex, domain=0:4, smooth, variable=\x, line width=0.4mm]
				plot (\x,4.5+\x/3);
				
			\draw[-, domain=0:1.5, smooth, variable=\x, line width=0.4mm]
				plot (\x,1.5-\x/3);
			\draw[-, domain=0:6, smooth, variable=\x, line width=0.4mm]
				plot (\x,1.5+\x/3);
			\draw[-, domain=0:6, smooth, variable=\x, line width=0.4mm]
				plot (\x,5.5-\x/3);
			\draw[-Latex, domain=0:1, smooth, variable=\x, line width=0.4mm]
				plot (\x,5.5+\x/3);
				
			\draw (1.5,1) node[dot, label={[label distance=1mm]below:{$(x_0,t_0)$}}]; % initial
			\draw (3,2) node {$\tfrac16$};
			\draw (3,3.95) node {$-\tfrac16$};
			\draw (1.5,5.5) node {$\tfrac16$};
			\draw (4.5,1) node {$0$};
			\draw (1.5,3) node {$0$};
			\draw (5.5,3) node {$0$};
			\draw (0.5,5) node {$0$};
			\draw (4.5,5) node {$0$};
		\end{tikzpicture}
		\caption{A slice of the Green's function associated with the wave operator $\L u = u_{tt}-9u_{xx}$ with initial point $(x_0,t_0)=(\tfrac14,\tfrac16)$. In this case, the Green's function is piecewise constant with jump discontinuities on the characteristics. The values of the Green's function are labeled in their respective regions.} 
		\label{fig:char-H-constant}
	\end{figure}
	
	\begin{example}
		Consider the constant-coefficient wave equation $u_{tt} - a^2u_{xx} = f$. The Green's function is quite simple to understand and can be written explicitly, but it would be very notationally complicated due to the boundary conditions. It is piecewise constant, and its value in each component is either $0$ or $\pm\tfrac1{2a}$ (see Figure \ref{fig:char-H-constant}). Its components are partitioned by characteristic curves, which are lines of slope $\tfrac1a$ emanating from the initial points $(x_0,t_0)$ and reflecting off the boundary. For general hyperbolic PDEs of the form \eqref{eq:problem}, the Green's function is usually not piecewise constant; nevertheless, it has the same qualitative properties, in particular jump discontinuities on the characteristics.
	\end{example}
	
	\section{Randomized SVD for Hilbert--Schmidt operators}\label{s:rSVD}
	
	This section introduces our primary tool, which we have adapted from randomized numerical linear algebra. The landmark result of \citet{Halko2011} proved that one could recover the column space of an unknown matrix with high accuracy and a high probability of success by multiplying it with standard Gaussian vectors. \citet{Boulle2022a} extend this result to HS operators and functions drawn from a non-standard Gaussian process. This algorithm is called the rSVD.

	\subsection{Randomized SVD}\label{ss:rSVD-op}
	
	Given a HS operator $\F:L^2(D_1)\to L^2(D_2)$ with SVE as in \eqref{eq:Ff-SVE}, we define two quasimatrices $\bU$ and $\bV$ containing the left and right singular functions of $\F$, so that the $j$th column of $\bU$ and $\bV$ is respectively $e_j$ and $v_j$. We also denote by $\bSigma$ the infinite diagonal matrix with the singular values $\sigma_1\ge\sigma_2\ge\cdots$ of $\F$ on the diagonal. For a fixed integer $k\ge0$, we define $\bV_1$ as the $D_1\times k$ quasimatrix whose columns are the first $k$ right singular functions $v_1,\dots,v_k$, and $\bV_2$ as the $D_1\times\infty$ quasimatrix whose columns are $v_{k+1},v_{k+2},\dots$. We analogously define $\bU_1$ and $\bU_2$ with the left singular functions. Similarly, we define $\bSigma_1$ as the $k\times k$ diagonal matrix with the first $k$ singular values $\sigma_1,\dots,\sigma_k$ on the diagonal, and $\bSigma_2$ as the $\infty\times\infty$ diagonal matrix with $\sigma_{k+1},\sigma_{k+2},\dots$ on the diagonal. In summary, we write
	$$\begin{blockarray}{ccc}
		& \\
		\begin{block}{c[cc]}
			\F = & \bU_1 & \bU_2 \\
		\end{block}
	\end{blockarray}
	\begin{blockarray}{cccc}
		k & \infty & & \\
		\begin{block}{[cc][c]c}
			\bSigma_1 & {\bf 0} & \bV_1^* & k \\
			{\bf 0} & \bSigma_2 & \bV_2^* & \infty \\
		\end{block}
	\end{blockarray}.$$
	Additionally, for some fixed continuous symmetric positive definite kernel $K:D_1\times D_1\to\R$ with eigenvalues $\lambda_1\ge\lambda_2\ge\dots>0$, we define the infinite matrix $\bC$ by
	\begin{equation}\label{eq:C}
		[\bC]_{ij} = \int_{D_1\times D_1} v_i(x)K(x,y)v_j(y)\dd x\dd y, \qquad i,j\ge1.
	\end{equation}
	Observe that $\bC$ is symmetric and positive definite, and that $\Tr(\bC) = \Tr(K) < \infty$, so it is a compact operator \citep[Lem.~1 and Eq.~(11)]{Boulle2022a}. We denote by $\bC^{-1}$ the inverse operator of $\bC$ on the domain for which it is well-defined. Furthermore, for fixed integer $k\ge1$, we partition $\bC$ into
	$$\bC = \begin{blockarray}{ccc}
		k & \infty & \\
		\begin{block}{[cc]c}
			\bC_{11} & \bC_{12} & k \\
			\bC_{21} & \bC_{22} & \infty \\
		\end{block}
	\end{blockarray}.\vspace{-3mm}$$
	Since $\bC_{11}$ is also positive definite and thus invertible, we define
	\begin{equation}\label{eq:gamma-eta}
		\gamma_k := \frac{k}{\lambda_1\Tr(\bC_{11}^{-1})}, \qquad 
		\xi_k := \frac{1}{\lambda_1\|\bC_{11}^{-1}\|},
	\end{equation}
	which quantify the quality of the covariance kernel $K$ with respect to $\F$. Indeed, the Courant--Fischer minimax principle implies that the $j$th largest eigenvalue of $\bC$ is bounded by $\lambda_j$, since $\bC$ is a principal submatrix of $\bV^*K\bV$. It follows that $0<\gamma_k,\xi_k\le 1$, and that the best case scenario occurs when the eigenfunctions of $K$ are the right singular functions of $\F$. In that case, $\bC$ is an infinite diagonal matrix with entries $\lambda_1\ge\lambda_2\ge\dots>0$, and $\gamma_k = k/(\sum_{j=1}^k\lambda_1/\lambda_j)$ and $\xi_k = \lambda_k/\lambda_1$ attain their minimal values. One can view $\xi_k$ as the operator norm analogue of $\gamma_k$, discussed in more detail in \citet[\S 3.4]{Boulle2022a}.
	
	Finally, for some $X\subset D_T$, let $\cR_X:L^2(D_T)\to L^2(X)$ denote the restriction to $X$. Its adjoint $\cR_X^*:L^2(X)\to L^2(D_T)$ is the zero extension operator, i.e., $\cR_X^*f$ is the function which is equal to $f$ on $X$ and equal to zero everywhere else. We denote by $\F_{X\times Y} := \cR_X\F\cR_Y^*$ the restriction of $\F$ to $X\times Y$. When considering the restricted operator, we define the analogous quantities
	\begin{equation}\label{eq:gamma-eta-XY}
		\gamma_{k,X\times Y} := \frac{k}{\lambda_1\Tr(\bC_{11,X\times Y}^{-1})}, \qquad
		\xi_{k,X\times Y} := \frac{1}{\lambda_1\|\bC_{11,X\times Y}^{-1}\|},
	\end{equation}
	where $\bC_{11,X\times Y}$ is the $k\times k$ matrix
	$$[\bC_{11,X\times Y}]_{ij} := \int_{D_1\times D_1} 
		\cR^*_Yv_{i,X\times Y}(x)K(x,y)\cR^*_Yv_{j,X\times Y}(y) \dd x \dd y,
		\qquad 1\le i,j\le k,$$
	and $v_{1,X\times Y},\dots,v_{k,X\times Y}$ are the dominant right $k$-singular functions of $\F_{X\times Y}$. We also define $\sigma_{1,X\times Y}\ge\sigma_{2,X\times Y}\ge\cdots$ as the singular values of $\F_{X\times Y}$.
	
	In this notation, we state an analogue of rSVD for HS operators \cite[Thm.~1]{Boulle2022a} with respect to the operator norm. We improve the error bound by a factor of $k$ compared to \citet[Thm.~1]{Boulle2022a}.
	
	\begin{theorem}\label{thm:rSVD-op}
		Let $D_1,D_2\subseteq\R^d$ be domains with $d\ge1$, and let $\F:L^2(D_1)\to L^2(D_2)$ be a HS operator with singular values $\sigma_1\ge\sigma_2\ge\dots\ge0$. Select a target rank $k\ge2$, an oversampling parameter $p\ge2$, and a $D_1\times(k+p)$ quasimatrix $\mathbf\Omega$ such that each column is independently drawn from $\GP(0,K)$, where $K:D_1\times D_1\to\R$ is a continuous symmetric positive definite kernel with eigenvalues $\lambda_1\ge\lambda_2\ge\dots>0$. Set $\bY=\F\mathbf\Omega$. Then 
		\begin{equation}\label{eq:rSVD-op-E}
			\E\|\F - \bP_\bY\F\| 
			\le \left(1 + \frac{1}{\xi_k} 
				+ \sqrt{\frac{\Tr(K)}{\lambda_1\xi_k}}\cdot \frac{e \sqrt{k+p}}{p}
				+ \sqrt{\frac{k}{\gamma_k(p+1)}}\,\right) \sigma_{k+1},
		\end{equation}
		where $\gamma_k,\xi_k$ are defined in \eqref{eq:gamma-eta}. Moreover, if $p\ge4$, then for any $s,t\ge1$, we have
		\begin{equation}\label{eq:rSVD-op-P}
			\|\F - \bP_\bY\F\| 
			\le \left[1 + \frac{1}{\xi_k} + \frac{e}{\sqrt{\xi_k}}
				\left(s + \sqrt{\frac{\Tr(K)}{\lambda_1}}\right) \frac{\sqrt{k+p}}{p+1} \cdot t
				+ \sqrt{\frac{k}{\gamma_k(p+1)}} \cdot t\,\right] \sigma_{k+1}
		\end{equation}
		with probability $\ge 1-2t^{-p}-e^{-s^2/2}$.
	\end{theorem}
	\begin{proof}
		See \hyperref[appendix-B]{Appendix B}.
	\end{proof}
	
	We also present a simplified version of the probability bound in \eqref{eq:rSVD-op-P}.
	
	\begin{corollary}\label{cor:rSVD-simple}
		Under the assumptions of Theorem \ref{thm:rSVD-op} with $k=p\ge 4$, we have
		\begin{equation}\label{eq:rSVD-simple-P}
			\|\F - \bP_\bY\F\|
			\le \left[1 + \frac{1}{\xi_k}\left(19 + 11\sqrt{\frac{\Tr(K)}{\lambda_1 k}}\right)\right]
				\sigma_{k+1}
		\end{equation}
		with probability $\ge 1 - 3e^{-k}$.
	\end{corollary}
	\begin{proof}
		We evaluate \eqref{eq:rSVD-op-P} by selecting $s=\sqrt{2k}$ and $t=e$. We also use the inequalities
		$$\frac{1}{\sqrt{\gamma_k}} \le \frac{1}{\sqrt{\xi_k}} \le \frac{1}{\xi_k},$$
		which follow from the fact that $0<\gamma_k,\xi_k\le 1$ and $\Tr(\bC_{11}^{-1}) \le k\|\bC_{11}^{-1}\|$.
	\end{proof}
	
	For convenience, we abbreviate the error factors in \eqref{eq:rSVD-op-P} and \eqref{eq:rSVD-simple-P} by
	\begin{align}
		A_{k,p}(s,t) &:= 1 + \frac{1}{\xi_k} + \frac{e}{\sqrt{\xi_k}}
			\left(s + \sqrt{\frac{\Tr(K)}{\lambda_1}}\right) \frac{\sqrt{k+p}}{p+1} \cdot t
			+ \sqrt{\frac{k}{\gamma_k(p+1)}} \cdot t, \label{eq:A} \\
		A_k &:= 1 + \frac{1}{\xi_k}\left(19 + 11\sqrt{\frac{\Tr(K)}{\lambda_1 k}}\right)
			\label{eq:A-simple}.
	\end{align}
	as well as the analogue
	\begin{equation}\label{eq:A-XY}
		A_{k,X\times Y} = 1 + \frac{1}{\xi_{k,X\times Y}}
			\left(19 + 11\sqrt{\frac{\Tr(K)}{\lambda_1 k}}\right)
	\end{equation}
	when considering a restricted domain $X\times Y\subset D_T\times D_T$.

	A ``power scheme'' version of rSVD, as described in \citet{Halko2011} and \citet{Rokhlin2009}, allows one to drive down the multiplicative factor in the probabilistic estimate \eqref{eq:rSVD-op-P} at the expense of a logarithmic factor increase in the number of operator-function products. The idea of the power scheme is that repeated projections of $\bY=\F\bOmega$ onto the column space of $\F$ improves the approximation.
	
	\begin{theorem}\label{thm:rSVD-power}
		Under the same assumptions as Theorem \ref{thm:rSVD-op}, select an integer $q\ge0$. Let $\sH=(\F\F^*)^q\F$ and set $\bZ=\sH\bOmega$. If $p\ge4$, then
		\begin{equation}\label{eq:rSVD-power-P}
			\|\F - \bP_\bZ\F\| \le A_{k,p}(s,t)^{1/(2q+1)} \sigma_{k+1}
		\end{equation}
		with probability $\ge 1 - 2t^{-p} - e^{-s^2/2}$, where $A_{k,p}(s,t)$ is defined in \eqref{eq:A}.
	\end{theorem}
	\begin{proof}
		See \hyperref[appendix-B]{Appendix B}.
	\end{proof}
	
	We summarize how the rSVD generates an approximant for the operator $\F$ in Algorithm \ref{alg:rSVD}, following \citet{Halko2011}.
	
	\begin{algorithm}
		\caption{Approximating $\F$ via rSVD}\label{alg:rSVD}
		\begin{algorithmic}[1]
			\Require HS operator $\F$, GP covariance kernel $K$, target rank $k\ge4$,
				oversampling parameter $p\ge2$, exponent $q\ge0$, additional parameters $s,t\ge1$
			\Ensure Approximation $\tilde\F$ of $\F$ within relative error given by 
				\eqref{eq:rSVD-power-P}
			\State Draw a $D_T\times(k+p)$ random quasimatrix $\bOmega$ with independent columns from $\GP(0,K)$ \\
				\Comment{See Section \ref{ss:GP} on how to draw such a quasimatrix}
			\State Construct $\bZ = (\F\F^*)^q\F\bOmega$ by multiplying with $\F$ and $\F^*$
			\State Compute the projector $\bP_\bZ$ via a QR factorization of $\bZ$ \\
				\Comment{See \citet{Trefethen2010} for Householder triangularization on quasimatrices}
			\State Form $\tilde\F = \bP_\bZ\F$
		\end{algorithmic}
	\end{algorithm} 
	
	\begin{remark}[Noisy training data]\label{rmk:noise}
		Here, we quantify the quality of the training data via the terms $\gamma_k$ and $\xi_k$, defined in \eqref{eq:gamma-eta}, which measure the deviation of the eigenspaces of the chosen covariance kernel $K$ from the dominant right singular subspaces of $\F$. One can also consider training data quality in the form of noise---namely, the presence of random additive perturbation errors arising from the computation of operator-function products, or from the collection of input-output data. In fact, it was proven in the finite-dimensional setting that the rSVD is stable with respect to noise in the input-output data \citep[Supp.~Info., \S 2.B]{Boulle2023}, so we assume noiseless data for simplicity.
	\end{remark}

	\subsection{Approximating singular values via singular subspaces}\label{ss:rSVA}
	
	Using the power scheme, we can not only improve the approximation of the operator $\F$ itself but also approximate its singular subspaces. Consequently, we can obtain excellent estimates for the singular values, which in conjunction with Theorem \ref{thm:rSVD-op} tells us about the numerical rank of $\F$. Notice that while Weyl's theorem \citep[Cor.~4.10]{Stewart1990} gives a straightforward bound on an operator's singular values, applying it meaningfully requires those singular values to decay quickly. Conversely, Theorem \ref{thm:rSVA} says that we can approximate the singular values regardless of their decay rate at the cost of additional operator-function products. We emphasize that our argument assumes the existence of a gap between adjacent singular values, although we believe this assumption is not necessary in principle.

	\begin{theorem}\label{thm:rSVA}
		Assume the hypotheses of Theorem \ref{thm:rSVD-op}, but with $\sigma_k>\sigma_{k+1}$. Select an integer $q\ge0$ and set $\delta_q$ as in \eqref{eq:rSVA-lem-q}. Let $\sH := (\F\F^*)^q\F$, $\bZ := \sH\bOmega$, and $\tilde\sH := \bP_\bZ\sH$. Let $\tilde\bU_k$ be the $D_1\times k$ quasimatrix whose columns are dominant left $k$-singular vectors of $\tilde\sH$. If $\delta_q A_{k,p}(s,t) < 1$, then
		\begin{equation}\label{eq:rSVA}
			\max_{1\le j\le k} |\sigma_j - \sigma_j(\tilde\bU_k^*\F)| 
				\le \frac{2\delta_q A_{k,p}(s,t)}{1 - \delta_q A_{k,p}(s,t)} \|\F\|
		\end{equation}
		with probability $\ge 1 - 2t^{-p} + e^{-s^2/2}$.
	\end{theorem}
	\begin{proof}
		See \hyperref[appendix-B]{Appendix B}.
	\end{proof}
	
	\section{Recovering the Green's function}\label{s:main}
	
	In this section, we construct a global approximant of the homogeneous Green's function of a 2-variable hyperbolic PDO $\L$ of the form \eqref{eq:L} in the domain $D_T=[0,1]\times[0,1]$. Recall that $\L$ is assumed to be linear strictly hyperbolic, and self-adjoint, with coefficients satisfying $a,b,c\in\cC^1(D_T)$. We suppose that one can generate $N$ forcing terms $\{f_j\}_{j=1}^N$ drawn from a Gaussian process with continuous symmetric positive definite covariance kernel $K:D_T\times D_T\to\R$ and use them to query the associated solution operator $\F$ of $\L$, as well as its adjoint $\F^*$, to generate solutions $u_j = \F f_j$ or $u_j = \F^*f_j$.\footnote{We note that $\F^*$ is the solution operator of the adjoint Cauchy problem of \eqref{eq:problem}, meaning that $u=\F^*f$ satisfies both the adjoint PDE $\L^*u = f$ as well as homogeneous conditions at the boundary and at the \emph{terminal} time $t=1$. In other words, the adjoint problem is the backward-in-time version of \eqref{eq:problem}. In particular, $\F$ is self-adjoint if and only if $\L$ is self-adjoint and the coefficients $a,b,c$ are time-independent.\label{fn:adjoint}} We derive a bound on the number of input-output pairs $\{(f_j,u_j)\}_{j=1}^N$ needed to approximate the Green's function within a given error tolerance measured in the operator norm, with a high probability of success. Our result is summarized in the following theorem.

	\begin{theorem}\label{thm:main}
		Let $\L$ be a hyperbolic PDO given in \eqref{eq:L}. Let $G:D_T\times D_T\to\R$ be the homogeneous Green's function of $\L$ in the domain $D_T$, and let $\F:L^2(D_T)\to L^2(D_T)$ be the solution operator with kernel $G$, as in \eqref{eq:G}. Additionally, define $\Xi_\epsilon$ and $\Psi_\epsilon$ as in \eqref{eq:Xi} and \eqref{eq:Psi}. For any sufficiently small $\epsilon>0$ such that $\Psi_\epsilon>0$, there exists a randomized algorithm that can construct an approximation $\tilde\F$ of $\F$ using $O(\Psi_\epsilon^{-1}\epsilon^{-7}\log(\Xi_\epsilon^{-1}\epsilon^{-1}))$ input-output training pairs of $\L$, such that
		$$\|\F - \tilde\F\| = O(\Xi_\epsilon^{-1}\epsilon)\|\F\|$$
		with probability $\ge 1 - O(e^{-1/\epsilon})$.
	\end{theorem}
	
	\begin{remark}[Increased regularity of coefficients]\label{rmk:regularity}
		The number of input-output training pairs can be reduced by assuming greater regularity of the coefficients $a,b,c$ of $\L$. In particular, if $a,b,c\in \cC^r(D_T)$ for some $r\ge1$, then  Theorem \ref{thm:main} easily generalizes, via \eqref{eq:approx-sv} and \eqref{eq:G-reg-boundary}, so that only $O(\Psi_\epsilon^{-1}\epsilon^{-(6+1/r)}\log(\Xi_\epsilon^{-1}\epsilon^{-1}))$ training pairs are needed to approximate $\F$ with relative error $O(\Xi_\epsilon^{-1}\epsilon)$. If we assume $a,b,c$ are analytic, the $\epsilon^{-1/r}$ factor can be reduced even further to $\log(\epsilon^{-1})$.
	\end{remark}
	
	A randomized algorithm that achieves Theorem \ref{thm:main} is summarized in Algorithm \ref{alg:main}, and we dedicate the remainder of this section to constructing it. We fix a sufficiently small $0<\epsilon<1/2$, so that by \eqref{eq:approx-sv} and \eqref{eq:G-reg-boundary}, there exists $k_\epsilon := \ceil{C\epsilon^{-1}} \ge 4$, where $C$ is a constant depending only on $\max_{(D_T\times D_T)\setminus Z}|\nabla G|$ and $Z$ is defined in \eqref{eq:Z-boundary}. This ensures that $\rank_\epsilon(\F_{X\times Y}) < k_\epsilon$ holds whenever $X\times Y\cap Z = \varnothing$. Here, $Z$ is the bundle of reflecting characteristic segments, defined in \eqref{eq:Z-boundary}. We also set the parameters $s=\sqrt{2k_\epsilon}$, $t=e$, and $p = k_\epsilon$, so that the probabilities of failure for Theorems \ref{thm:rSVD-op}, \ref{thm:rSVD-power}, and \ref{thm:rSVA} are bounded by $3e^{-k_{\epsilon}}$ (see Corollary \ref{cor:rSVD-simple}).
	
	\begin{algorithm}
		\caption{Learning the solution operator via input-output data}\label{alg:main}
		\begin{algorithmic}[1]
			\Require Black-box solver for $\L$, GP covariance kernel $K$, 
				tolerance $0 < \epsilon < 1/2$
			\Ensure Approximation $\tilde\F$ of the solution operator $\F$ within relative error 
				$\Xi_\epsilon^{-1}\epsilon$
			\State Set target rank $k_\epsilon \ge 4$
			\While{$\vol(D_\red(L)) > \epsilon^2$} 
				\For{$D \in \cD_\red(L)$}
					\State Partition $D$ into 16 subdomains $D_1,\dots,D_{16}$
						\Comment{Section \ref{ss:partition}}
					\For{$i=1:16$}
						\State Determine the numerical rank of $\F$ on $D_i$ (Algorithm 
							\ref{alg:rank-detection})
							\Comment{Section \ref{ss:rank-detection}}
						\If{$\F$ is numerically low-rank on $D_i$}
							\State Color $D_i$ green
							\State Approximate $\F$ on $D_i$ using the rSVD (Algorithm \ref{alg:rSVD})
								\Comment{Section \ref{ss:adm-local}}
						\Else
							\State Color $D_i$ red and add $D_i$ to $\cD_\red(L+1)$
						\EndIf
					\EndFor
				\EndFor
			\EndWhile
			\State Pad the approximant $\tilde\F$ with zeros on the remaining red subdomains
				\Comment{Section \ref{sss:weyl}}
		\end{algorithmic}
	\end{algorithm}
	
	\subsection{Rank detection scheme}\label{ss:rank-detection}
	
	Since the locations of the characteristics are unknown, we adaptively partition $D_T\times D_T$ in a hierarchical manner. The adaptive part relies on a ``rank detection scheme'' that detects the numerical rank of $\F$ in a given subdomain $X\times Y\subset D_T\times D_T$ (see Algorithm \ref{alg:rank-detection}).
	
	\begin{algorithm}
		\caption{Detecting the numerical rank of $\F$ in a subdomain}\label{alg:rank-detection}
		\begin{algorithmic}[1]
			\Require Black-box solver for $\L$, GP covariance kernel $K$, subdomain $X\times Y$, 
				tolerance $0 < \epsilon < 1$
			\Ensure Classification of numerical rank of $\F_{X\times Y}$
			\State Set target rank $k_\epsilon \ge 4$
			\State Set exponent $q_{\epsilon,X\times Y} \approx \log(\epsilon^{-1})$
			\State Construct $\tilde\sH_{X\times Y}$ using the rSVD (Algorithm \ref{alg:rSVD}) with 
				exponent parameter 1 applied to $\sH = (\F\F^*)^{q_\epsilon,X\times Y}\F$ 
			\State Compute a SVD of $\tilde\sH_{X\times Y}$ to obtain the dominant left 
				$k_\epsilon$-singular functions $\tilde\bU_{k_\epsilon,X\times Y}$
			\State Construct $\tilde\bU^*_{k_\epsilon,X\times Y}\F_{X\times Y}$ using the 
				black-box solver for $\L$
			\State Compute the singular values $\hat\sigma_{1,X\times Y} \ge \dots \ge
				\hat\sigma_{k_\epsilon,X\times Y}$ of $\tilde\bU^*_{k_\epsilon,X\times Y}\F_{X\times Y}$
			\If{$\hat\sigma_{k_\epsilon,X\times Y} < 4\epsilon\hat\sigma_{1,X\times Y}$}
				$\rank_{5\epsilon}(\F_{X\times Y}) < k_\epsilon$
			\Else \ $\rank_\epsilon(\F_{X\times Y}) \ge k_\epsilon$
			\EndIf
		\end{algorithmic}
	\end{algorithm}
	
	Similar to \citet[\S 4.1.2]{Boulle2022a}, we generate a $Y\times 2k_\epsilon$ quasimatrix $\bOmega$ such that each column is independently drawn from a Gaussian process defined on $Y$, given by $\GP(0,\cR_{Y\times Y}K)$, where $\cR_{Y\times Y}$ is the operator that restricts functions to the domain $Y\times Y$. Then we extend by zero each column of $\bOmega$ from $L^2(Y)$ to $L^2(D_T)$ in the form $\cR_Y^*\bOmega$. 
	
	Next, we select
	\begin{equation}\label{eq:q}
		q_{\epsilon,X\times Y} 
			:= \max\left(0,\left\lceil\frac12 \left(\frac{\log(1+A_{k_\epsilon,X\times Y}
			+ 2A_{k_\epsilon,X\times Y}/\epsilon)}
			{\log(\sigma_{k_\epsilon,X\times Y}/\sigma_{k_\epsilon+1,X\times Y})} - 1
			\right)\right\rceil\right),
	\end{equation}
	where $A_{k_\epsilon,X\times Y}$ is defined in \eqref{eq:A-XY}, and set $\sH = (\F\F^*)^{q_{\epsilon,X\times Y}}\F$. We approximate the range of $\sH$ via $\bZ = \sH\cR_Y^*\bOmega$ and then compute the rank-$2k_\epsilon$ approximant $\tilde\sH_{X\times Y} = \bP_{\cR_X\bZ}\cR_X\sH\cR_Y^*$. Since $\tilde\sH_{X\times Y}$ has finite rank, we can compute its SVD and extract the quasimatrix $\tilde\bU_{k_\epsilon,X\times Y}$ whose columns are the dominant left $k_\epsilon$-singular vectors of $\tilde\sH_{X\times Y}$, and finally compute the dominant singular values $\hat\sigma_{1,X\times Y}\ge\dots\ge\hat\sigma_{k_\epsilon,X\times Y}$ of the operator $\tilde\bU_{k_\epsilon,X\times Y}^*\F_{X\times Y}$. In total, these computations require $k_\epsilon(8q_{\epsilon,X\times Y}+5)$ input-output training pairs, and rely on the querying the adjoint of $\F$.
	
	Our choice of $q_{\epsilon,X\times Y}$ in \eqref{eq:q} is motivated as follows. Given $\delta_{q_{\epsilon,X\times Y}}$ as defined in \eqref{eq:rSVA-lem-q}, we have
	$$\frac{2\delta_{q_{\epsilon,X\times Y}} A_{k_\epsilon,X\times Y}}
		{1 - \delta_{q_{\epsilon,X\times Y}} A_{k_\epsilon,X\times Y}} \le \epsilon,$$
	which in conjunction with Theorem \ref{thm:rSVA} implies that the inequalities $\hat\sigma_{k_\epsilon,X\times Y} \le \sigma_{k_\epsilon,X\times Y} + \epsilon\sigma_{1,X\times Y}$ and $(1-\epsilon)\sigma_{1,X\times Y} \le \hat\sigma_{1,X\times Y}$ hold with probability $\ge 1 - 3e^{-k_\epsilon}$. If $\rank_\epsilon(\F_{X\times Y}) < k_\epsilon$, then $\sigma_{k_\epsilon,X\times Y} \le \epsilon\sigma_{1,X\times Y}$, hence
	$$\hat\sigma_{k_\epsilon,X\times Y} \le 2\epsilon\sigma_{1,X\times Y}
		\le \frac{2\epsilon}{1-\epsilon}\hat\sigma_{1,X\times Y}.$$
	Since $\epsilon<1/2$, then the right-hand side is bounded by $4\epsilon\hat\sigma_{1,X\times Y}$. On the other hand, if $\hat\sigma_{k_\epsilon,X\times Y} \le 4\epsilon\hat\sigma_{1,X\times Y}$, then Theorem \ref{thm:rSVA} again yields
	$$\sigma_{k_\epsilon,X\times Y} 
		\le \hat\sigma_{k_\epsilon,X\times Y} + \epsilon\sigma_{1,X\times Y}
		\le 5\epsilon\sigma_{1,X\times Y}$$
	hence $\rank_{5\epsilon}(\F_{X\times Y}) < k_\epsilon$. In summary, we have shown that
	\begin{equation}\label{eq:adm-chain}
	 	\rank_\epsilon(\F_{X\times Y}) < k_\epsilon \ \implies\ 
		\hat\sigma_{k_\epsilon,X\times Y} \le 4\epsilon\hat\sigma_{1,X\times Y} \ \implies\ 
		\rank_{5\epsilon}(\F_{X\times Y}) < k_\epsilon
	\end{equation}
	holds with probability $\ge 1 - 3e^{-k_\epsilon}$.
	
	The preceding argument implies that we need only to check the validity of the inequality 
	\begin{equation}\label{eq:adm-test}
		\hat\sigma_{k_\epsilon,X\times Y} < 4\epsilon\hat\sigma_{1,X\times Y}
	\end{equation}
	to determine with high probability whether or not $\F_{X\times Y}$ has numerical rank bounded by $k_\epsilon$, with low error tolerance. In particular, every subdomain that does not intersect $Z$ has $\rank_\epsilon(\F_{X\times Y}) < k_\epsilon$, and the test passes on all such subdomains. Conversely, even if \eqref{eq:adm-test} happens to be satisfied for a subdomain that does intersect $Z$, then that subdomain still has a low numerical rank with comparatively small error tolerance, namely, $5\epsilon$.

	\subsection{Hierarchical partition of domain}\label{ss:partition}
	
	We now describe the hierarchical partitioning of the domain $D_T\times D_T = [0,1]^4$, so that $\F$ is numerically low-rank on many subdomains, while the subdomains on which $\F$ is not low-rank have a small total volume. Since the probability of failure is very low, we describe the partition deterministically and relegate the discussion of the probabilistic aspects to Section \ref{sss:recovery-P}. In particular, we assume throughout this section that \eqref{eq:adm-chain} holds deterministically.
	
	Our partitioning strategy proceeds as follows. At level $L=0$, we consider only one subdomain, which is the entirety of $D_T\times D_T$. On each subdomain at the current partition level, we perform the rank detection procedure described in Section \ref{ss:rank-detection} to check if $\F$ is numerically low-rank on the subdomain. If it is, we color the subdomain green and approximate $\F$ restricted to the subdomain using the rSVD. Otherwise, we color the subdomain red. The next level of the partition then consists of partitioning each red subdomain into smaller subdomains and repeating the coloring process. Since the characteristics lie only in red subdomains, then at each level of the partition we learn the location of the characteristics with finer detail (see Figure \ref{fig:hierarchy}).
	
	The hierarchical partition of $D_T\times D_T = [0,1]^4$ for $n$ levels is defined recursively as follows.
	\begin{itemize}
		\item At level $L=0$, the domain $I_{1,1,1,1} := I_1\times I_1\times I_1\times I_1 = [0,1]^4$ is the root of the partitioning tree.
		\item At a given level $0\le L\le n-1$, a node $I_{j_1,j_2,j_3,j_4} := I_{j_1}\times I_{j_2}\times I_{j_3}\times I_{j_4}$ is colored red if \eqref{eq:adm-test} fails to hold for $X\times Y = I_{j_1,j_2,j_3,j_4}$. Otherwise, the node $I_{j_1,j_2,j_3,j_4}$ is colored green. Green nodes have no children, whereas red nodes have $2^4$ children defined as
		$$\{I_{2j_1+k_1}\times I_{2j_2+k_2}\times I_{2j_3+k_3}\times I_{2j_4+k_4} : 
			k_i\in\{0,1\},\  i=1,2,3,4\}.$$
		Here, if $I_j=[a,b]$, $0\le a<b\le 1$, then we define $I_{2j}=[a,\frac{a+b}{2}]$ and $I_{2j+1}=[\frac{a+b}{2},b]$.
	\end{itemize}
	Every subdomain $X\times Y$ in levels $0\le L\le n-1$ is green and satisfies \eqref{eq:adm-test}, hence $\rank_{5\epsilon}(\F_{X\times Y}) < k_\epsilon$, by \eqref{eq:adm-chain}. At the end, we perform the test \eqref{eq:adm-test} for all remaining subdomains at level $n$. Those for which \eqref{eq:adm-test} holds are colored green, while the rest are colored red.

	\subsection{Local approximation of the Green's function}\label{ss:adm-local}
	
	The hierarchical partition of $D_T\times D_T$ described above tells us where $\F$ is numerically low-rank and where it is not. If a subdomain $X\times Y$ is colored green, that is, a subdomain on which \eqref{eq:adm-test} holds, then we can approximate it using the rSVD. Since we have generated $\bZ = \sH\cR_Y^*\bOmega$ in the process of rank detection, where $\sH = (\F\F^*)^{q_{\epsilon,X\times Y}}\F$ and $\bOmega$ is a quasimatrix with columns drawn from $\GP(0,\cR_{Y\times Y}K)$, then we only require an additional $2k_\epsilon$ input-output training pairs to construct $\tilde\F_{X\times Y} := \bP_{\cR_X\bZ}\cR_X\F\cR_Y^*$.\footnote{In other words, we can skip lines 2--3 in Algorithm \ref{alg:rSVD}.} By Theorem \ref{thm:rSVD-power} and \eqref{eq:adm-chain}, we have
	\begin{equation}\label{eq:local-approx}
		\|\F_{X\times Y} - \tilde\F_{X\times Y}\| 
			\le A_{k_\epsilon,X\times Y}^{1/(2q_{\epsilon,X\times Y}+1)}\sigma_{k_\epsilon+1,X\times Y}
			\le 5\epsilon A_{k_\epsilon,X\times Y}^{1/(2q_{\epsilon,X\times Y}+1)}\|\F_{X\times Y}\|.
	\end{equation}
	Moreover, this error estimate holds with probability 1, conditioned on the event that \eqref{eq:adm-chain} is valid (see Section \ref{sss:recovery-P}).
	
	On the other hand, if \eqref{eq:adm-test} does not hold on $X\times Y$, then we approximate $\F$ by zero. Since the Green's function $G$is continuous in the compact domain $D_T\times D_T$, except for jumps along the characteristics and the diagonal, then there exists a constant $C$ such that
	$$G(x,t;x_0,t_0) \le C\|\F\|, \qquad (x,y,x_0,y_0)\in D_T\times D_T.$$
	Hence, for any $X,Y\subset D_T$, we have
	\begin{equation}\label{eq:HS-op-estimate}
		\|G\|_{L^2(X\times Y)}^2 \le C^2 \vol(X\times Y)\|\F\|^2
	\end{equation}
	and thus
	\begin{equation}\label{eq:inadm-approx}
		\|\F_{X\times Y}\| \le C\sqrt{\vol(X\times Y)}\|\F\|,
	\end{equation}
	where $C$ does not depend on $X\times Y$. Therefore, for sufficiently small domains $X\times Y$, we may approximate $\F_{X\times Y}$ by zero. At the end of partitioning, the total volume of the subdomains on which \eqref{eq:adm-test} fails to hold is negligible (see Section \ref{sss:weyl}).

	\subsection{Recovering the Green's function on the entire domain}
	We now show that we can recover $\F$ on the entire domain $D_T\times D_T$.

	\subsubsection{Global approximation near characteristics}\label{sss:weyl}
	
	First, we ensure that the volume of all subdomains where $\F$ is not low-rank is small enough to safely ignore that part of $\F$ by approximating it by zero. As one increases the level of hierarchical partitioning, the volume of such subdomains shrinks to zero.
	
	\begin{figure}[t]
		\centering
		\begin{subfigure}{0.45\textwidth}
			\begin{adjustbox}{max width=\textwidth}
			\begin{tikzpicture}[
				dot/.style={circle, fill=black, inner sep=0pt, minimum size=5pt, node contents={}},
				g/.style={fill=green+}, r/.style={fill=red+}]
				
				% colored grid
				\matrix[matrix of nodes, nodes={draw, minimum size=1.5cm}, 
					nodes in empty cells, column sep=-0.4pt, row sep=-0.4pt,
					xshift=3cm, yshift=3cm](M){
					|[r]| & |[r]| & |[r]| & |[g]| \\
					|[r]| & |[r]| & |[r]| & |[r]| \\
					|[r]| & |[r]| & |[g]| & |[r]| \\
					|[g]| & |[r]| & |[r]| & |[r]| \\
				};
				
				% axes
  				\draw[-latex] (0,0) -- (6.5,0) node[right] {$x$};	% x-axis (t=0)
				\draw[-] (0,6) -- (6,6);							% x-axis (t=1)
				\draw[-latex] (0,0) -- (0,6.5)	node[above] {$t$};	% t-axis (x=0)
				\draw[-] (6,0) -- (6,6);							% t-axis (x=1)
				
				% characteristics
				\begin{scope}
					\clip (0,0) rectangle (6,6); % clip edges to domain
					
					% positive characteristic
					\draw[scale=6, domain=0.1:0.377, smooth, variable=\y, line width=0.4mm] 
						plot ({sinh(-0.903 + 2*sqrt(\y+1))-1}, \y);		% initial
					\draw[scale=6, domain=0.377:1.1, smooth, variable=\y, line width=0.4mm]
						plot ({-sinh(-3.791 + 2*sqrt(\y+1))-1}, \y);	% reflect
					
					% negative characteristic
					\draw[scale=6, domain=0.1:0.4528, smooth, variable=\y, line width=0.4mm]
						plot ({-sinh(-3.292 + 2*sqrt(\y+1))-1}, \y);	% initial
					\draw[scale=6, domain=0.4528:1.1, smooth, variable=\y, line width=0.4mm]
						plot ({sinh(-1.529 + 2*sqrt(\y+1))-1}, \y);		% reflect
				\end{scope}
				
				% initial point
				\draw (3,0.6) node[dot, label={[label distance=-1mm]below:{$(x_0,t_0)$}}];
			\end{tikzpicture}
			\end{adjustbox}
			\subcaption{}\label{fig:hierarchy-1}
		\end{subfigure}
		\hfill
		\begin{subfigure}{0.45\textwidth}
			\begin{adjustbox}{max width=\textwidth}
			\begin{tikzpicture}[
				dot/.style={circle, fill=black, inner sep=0pt, minimum size=5pt, node contents={}},
				pics/gcell/.style n args={4}{code={ % args = {coord}{size}{line width}{color}
					\begin{scope}
						\clip[postaction={fill=#4, line width=#3}]
							#1 rectangle+(#2,#2);
						\draw[line width=0.8mm, dash=on 8pt off 6pt phase 2pt] % tube cap
							(4.1838,6) circle (1.05cm+0.4mm);
						\draw[scale=6, smooth, variable=\y, line width=2.1cm+1.6mm,
							dash=on 8pt off 6pt phase 2pt] % tube border
							plot[domain=0.1:0.377] ({sinh(-0.903 + 2*sqrt(\y+1))-1}, \y)
							plot[domain=0.1:0.4528] ({-sinh(-3.292 + 2*sqrt(\y+1))-1}, \y)
							plot[domain=0.377:1] ({-sinh(-3.791 + 2*sqrt(\y+1))-1}, \y)
							plot[domain=0.4528:1] ({sinh(-1.529 + 2*sqrt(\y+1))-1}, \y);
						\draw[green+, scale=6, smooth, variable=\y, line width=2.1cm, 
							line cap=round] % tube main
							plot[domain=0.1:0.377] ({sinh(-0.903 + 2*sqrt(\y+1))-1}, \y)
							plot[domain=0.1:0.4528] ({-sinh(-3.292 + 2*sqrt(\y+1))-1}, \y)
							plot[domain=0.377:1] ({-sinh(-3.791 + 2*sqrt(\y+1))-1}, \y)
							plot[domain=0.4528:1] ({sinh(-1.529 + 2*sqrt(\y+1))-1}, \y);	
					\end{scope}
				}}]
				
				% axes
  				\draw[-latex] (0,0) -- (6.5,0) node[right] {$x$};	% x-axis (t=0)
				\draw[-] (0,6) -- (6,6);							% x-axis (t=1)
				\draw[-latex] (0,0) -- (0,6.5) node[above] {$t$};	% t-axis (x=0)
				\draw[-] (6,0) -- (6,6);							% t-axis (x=1)
				
				\clip (0,0) rectangle (6,6); % clip edges in domain
				
				% colored grid
				%%% red/green binary map
				\pgfplotstableread[col sep=comma]{
					0,0,0,1,0,0,1,1
					1,1,0,0,0,1,1,1
					1,0,0,0,0,1,1,1
					0,0,1,1,0,0,0,1
					0,1,1,1,1,1,0,0
					0,0,0,1,1,1,0,0
					1,1,0,0,0,0,0,1
					1,1,1,0,0,1,1,1
				}\grid
				
				%%% tube
  				\pgfplotstableforeachcolumn\grid\as\col {
    				\pgfplotstableforeachcolumnelement{\col}\of\grid\as\colcnt {
						\ifnum\colcnt=1
							\pic at (0,0) 
								{gcell={($(0,5.25)-\pgfplotstablerow*(0,0.75)+\col*(0.75,0)$)}
								{0.75}{0}{green+}};
						\fi
						\ifnum\colcnt=0
							\draw[fill=red+] ($(0,5.25)-\pgfplotstablerow*(0,0.75)+\col*(0.75,0)$)
								rectangle+(0.75,0.75);
						\fi
    				}
  				}
								
				%%% grid lines
				\draw[step=0.75, line width=0.4pt] (0,0) grid (6,6);
				
				%%% previously green nodes
				\pic at (0,0) {gcell={(0.2pt,0.2pt)}{1.5cm-0.4pt}{0.4pt}{green+}};
				\pic at (0,0) {gcell={(3cm+0.2pt,1.5cm+0.2pt)}{1.5cm-0.4pt}{0.4pt}{green+}};
				\pic at (0,0) {gcell={(4.5cm+0.2pt,4.5cm+0.2pt)}{1.5cm-0.4pt}{0.4pt}{green+}};
				%\draw[fill=green+, line width=0.4pt] (0,0) rectangle+(1.5,1.5);
				%\node[g, minimum size=1.5cm-0.4pt] at (3.75,2.25) {};
				%\node[g, minimum size=1.5cm-0.4pt] at (5.25,5.25) {};

				% characteristics
				% positive characteristic
				\draw[scale=6, domain=0.1:0.377, smooth, variable=\y, line width=0.4mm] 
					plot ({sinh(-0.903 + 2*sqrt(\y+1))-1}, \y);		% initial
				\draw[scale=6, domain=0.377:1.1, smooth, variable=\y, line width=0.4mm]
					plot ({-sinh(-3.791 + 2*sqrt(\y+1))-1}, \y);	% reflect
					
				% negative characteristic
				\draw[scale=6, domain=0.1:0.4528, smooth, variable=\y, line width=0.4mm]
					plot ({-sinh(-3.292 + 2*sqrt(\y+1))-1}, \y);	% initial
				\draw[scale=6, domain=0.4528:1.1, smooth, variable=\y, line width=0.4mm]
					plot ({sinh(-1.529 + 2*sqrt(\y+1))-1}, \y);		% reflect
				
				% initial point
				\draw (3,0.6) node[dot, label={[label distance=-1mm]below:{\,\,$(x_0,t_0)$}}];
			\end{tikzpicture}
			\end{adjustbox}
			\subcaption{}\label{fig:hierarchy-2}
		\end{subfigure}
		\caption{Slices through $D_T\times D_T$ at two levels of hierarchical partitioning. A subdomain intersecting a characteristic is colored red and is further partitioned at the next level. Otherwise, it is colored green. The boundary of the tube of \eqref{eq:tube} is depicted by the thick dashed lines in (b); notice that every red subdomain is contained within the tube's interior.}
		\label{fig:hierarchy}
	\end{figure}
	
	Let $\cD_\red(L)$ denote the collection of subdomains $X\times Y$ at level $L$ that are colored red, such that $\rank_{5\epsilon}(\F_{X\times Y})\ge k_\epsilon$. Let $D_\red(L) := \bigcup_{D\in\cD_\red(L)} D$. By \eqref{eq:G-reg-boundary}, \eqref{eq:adm-chain} and the definition of $k_\epsilon$, we have $X\times Y\in\cD_\red(L)$ only if $(X\times Y)\cap Z \neq \varnothing$. Subdomains at level $L$ are 4-dimensional cubes with side length $2^{-L}$, so every $X\times Y\in\cD_\red(L)$ is contained in the closed tubular neighborhood of $Z$ in $D_T\times D_T$ with radius $2^{-L+1}$, defined as
	\begin{equation}\label{eq:tube}
		T(Z,2^{-L+1}) := \{p\in D_T\times D_T : \dist(p,Z) \le 2^{-L+1}\}.
	\end{equation}
	The volume of $T(Z,2^{-L+1})$ depends only on the coefficient $a$ of \eqref{eq:L} and can be computed explicitly using a Weyl-type tube formula.\footnote{For $a\in\cC^2$, the classical Weyl tube formula suffices, whereas for $a$ with less regularity one requires more complicated expressions; see, e.g., \citet[Thm.~5.6]{Federer1959}, \citet[Thm.~4.8]{Gray2003}, and \citet[Eq.~(4.4)]{Hug2004}. We omit the details.} In particular, we have
	\begin{equation}\label{eq:vol-hi-asymp}
		\vol(D_\red(L)) \le \vol(T(Z,2^{-L+1})) = O(2^{-L}),
	\end{equation}
	hence
	$$\sum_{X\times Y\in\cD_\red(L)}\|\F_{X\times Y}\| \le O(2^{-L/2}) \|\F\|$$
	by \eqref{eq:inadm-approx}. Therefore, by selecting $n_\epsilon:=\lceil\log_2(1/\epsilon^2)\rceil$, we guarantee that 
	\begin{equation}\label{eq:inadm-global}
		\sum_{X\times Y\in\cD_\red(n_\epsilon)}\|\F_{X\times Y}\| = O(\epsilon)\|\F\|
	\end{equation}
	and can safely approximate $\F$ by zero on $D_\red(n_\epsilon)$.

	\subsubsection{Recovery rate}\label{sss:recovery-rate}
	
	Input-output training pairs are required both for rank detection on subdomains and for approximation for the subdomains on which $\F$ is numerically low-rank. We count the training pairs needed for each task separately.

	Since rank detection occurs on every subdomain in the partitioning tree, let $\cD_\tree(n_\epsilon)$ denote the entire collection of subdomains in the tree after $n_\epsilon$ levels of partitioning. For each  subdomain colored red at a given level $L$, we split it into 16 smaller subdomains and perform the rank detection scheme on each, hence
	$$\#\cD_\tree(n_\epsilon) = 16 \sum_{L=0}^{n_\epsilon-1} \#\cD_\red(L).$$
	Subdomains at the $L$th level of partitioning have volume $16^{-L}$, so it follows from \eqref{eq:vol-hi-asymp} that
	$$\#\cD_\red(L) \le 16^L\vol(D_\red(L)) = O(8^L).$$
	Therefore,
	\begin{equation}\label{eq:num-domains}
		\#\cD_\tree(n_\epsilon) = O(8^{n_\epsilon}) = O(\epsilon^{-6}).
	\end{equation}
	Finally, recall from Section \ref{ss:rank-detection} that performing the rank detection scheme on a subdomain $X\times Y$ requires $k_\epsilon(8q_{\epsilon,X\times Y}+5)$ training pairs. We remove the dependence on $X\times Y$ by introducing
	\begin{equation}\label{eq:Xi}
		\Xi_\epsilon := \min\{\xi_{k_\epsilon,X\times Y} : X\times Y\in\cD_\tree(n_\epsilon)\}
			\in (0,1],
	\end{equation}
	which represents the quality of the covariance kernel $K$, as well as
	\begin{equation}\label{eq:Psi}
		\Psi_\epsilon := \min\left\{\log\left(\frac{\sigma_{k_\epsilon,X\times Y}}
			{\sigma_{k_\epsilon+1,X\times Y}}\right) : X\times Y\in\cD_\tree(n_\epsilon)\right\},
	\end{equation}
	which represents the influence of the singular value gaps in the rank detection scheme. It follows from \eqref{eq:A-XY} that $A_{k_\epsilon,X\times Y} = O(\Xi_\epsilon^{-1})$ and thus $q_{\epsilon,X\times Y} = O(\Psi_\epsilon^{-1}\log(\Xi_\epsilon^{-1}\epsilon^{-1}))$ by \eqref{eq:q}. Therefore, the total number of training pairs needed for rank detection on every domain in $\cD_\tree(n_\epsilon)$ is given by
	$$N_\epsilon^{(1)} 
		:= \sum_{X\times Y\in\cD_\tree(n_\epsilon)} k_\epsilon(8q_{\epsilon,X\times Y}+5)
		= O(\Psi_\epsilon^{-1}\epsilon^{-7}\log(\Xi_\epsilon^{-1}\epsilon^{-1})).$$
	
	To count training pairs needed for approximation on domains where $\F$ is numerically low-rank, we first observe that
	$$\#\cD_\green(n_\epsilon) 
		= \sum_{L=1}^{n_\epsilon}\Big(16\cdot \#\cD_\red(L-1) - \#\cD_\red(L)\Big) 
		= O(8^{n_\epsilon}),$$
	where $\cD_\green(n_\epsilon)$ is the collection of subdomains $X\times Y$ in the hierarchical level $n_\epsilon$ colored green, i.e., those satisfying $\rank_{5\epsilon}(\F_{X\times Y}) < k_\epsilon$. Recall from Section \ref{ss:adm-local} that only $2k_\epsilon$ additional training pairs are required to generate an approximant on a subdomain via rSVD after performing rank detection, so the number of training pairs needed for approximation is given by
	$$N_\epsilon^{(2)} 
		:= \sum_{X\times Y\in\cD_\green(n_\epsilon)} 2k_\epsilon = O(\epsilon^{-7}).$$
	We conclude that the total number of training pairs required is
	$$N_\epsilon := N_\epsilon^{(1)} + N_\epsilon^{(2)}
		= O(\Psi_\epsilon^{-1}\epsilon^{-7}\log(\Xi_\epsilon^{-1}\epsilon^{-1})).$$

	\subsubsection{Global approximation error}\label{sss:global-error}
	
	Let $\tilde\F$ be the approximant to $\F$ given by stitching together, at the end of the hierarchical partition, the approximants $\tilde\F_{X\times Y}$ generated on green subdomains $X\times Y$, and the zero operator on red subdomains. By \eqref{eq:inadm-global}, we have
	$$\|\F - \tilde\F\| \le \sum_{X\times Y\in\cD_\green(n)} \|\F_{X\times Y} - \tilde\F_{X\times Y}\| 
		+ O(\epsilon)\|\F\|,$$
	where every term in the summation satisfies \eqref{eq:local-approx}. Recall that $A_{k_\epsilon,X\times Y}=O(\Xi_\epsilon^{-1})$, hence
	$$\|\F - \tilde\F\| \le O(\Xi_\epsilon^{-1}\epsilon)\|\F\|$$
	is our final error.

	\subsubsection{Recovery probability}\label{sss:recovery-P}
	
	Thus far, we have assumed that our rank detection is guaranteed to succeed on each domain, i.e., that \eqref{eq:adm-chain} holds deterministically. Circumventing this assumption is simple: since the forcing terms in the input-output training data are independent, and since we perform rank detection on $\#\cD_\tree(n_\epsilon) = O(\epsilon^{-6})$ domains, we have
	$$\P\Big\{\text{\eqref{eq:adm-chain} holds for every } X\times Y\in \cD_\tree(n)\Big\}
			\ge (1 - 3e^{-k_\epsilon})^{\#\cD_\tree(n_\epsilon)} 
			= 1 - O(e^{-1/\epsilon}).$$
	Notice that for each subdomain, the approximant is generated using the same test quasimatrices as used in the rank detection scheme, hence all error bounds for approximants, e.g., \eqref{eq:local-approx}, are guaranteed to hold when conditioned on the validity of \eqref{eq:adm-chain}. Therefore, our global probability of failure is $O(e^{-1/\epsilon})$, and we have completed the proof of Theorem \ref{thm:main}.
	
	\section{Numerical Example}\label{s:numerics}
	
	We implement Algorithms \ref{alg:rSVD}, \ref{alg:main}, and \ref{alg:rank-detection} in MATLAB for the constant coefficient wave operator $\L u = u_{tt} - 4u_{xx}$ with homogeneous initial and boundary conditions.\footnote{MATLAB code is available at \url{https://github.com/chriswang030/OperatorLearningforHPDEs}.}\textsuperscript{,}\footnote{While our code has no difficulty handling more complicated hyperbolic PDEs, the Green's functions for equations with variable coefficients are generally unknown analytically, hence training data must be generated from, say, a numerical solver. Experiments with various popular techniques showed that such solvers were either (a) too slow to be practical, given the large amount of data needed, or (b) introduced numerical dissipation and/or dispersion that dominated the approximation error and resulted in a poor Green's function approximation. We discuss these issues in more detail in Section \ref{ss:data}.} Our implementation discretizes the domain blockwise at Gauss--Legendre quadrature nodes. Input-output data was generated using the known analytical expression for the true Green's function.
	
	Figure \ref{fig:numerics} depicts a slice of the approximate Green's function constructed by Algorithm \ref{alg:main} after eight levels of partitioning, as well as the pointwise error in comparison with the actual Green's function. We observe that the rank detection scheme successfully distinguishes between those subdomains that intersect characteristics and those that do not, hence the partition is fine in areas near characteristics and coarse in areas away from characteristics. Errors are concentrated near the characteristics, as expected, while low-rank subdomains are approximated very well. Figure \ref{fig:converge} shows the experimental rate of convergence as the amount of input-output training data increases. As predicted by Theorem \ref{thm:main}, the error decays like $O(N^{-1/7})$, where $N$ is the number of training data pairs. The slow convergence rate---theoretically predicted and empirically observed---shows that large amounts of data may be needed to accurately approximate Green's functions associated with hyperbolic PDEs. We emphasize that our algorithm is guarding against the nastiest possible solution operators associated with a hyperbolic PDE; if one assumes in advance that the PDE has constant coefficients, then one can do dramatically better. Nevertheless, the numerical example demonstrates that our method is stable and robust to discretization errors.
	
	\begin{figure}[t]
		\centering
		\includegraphics[width=\textwidth]{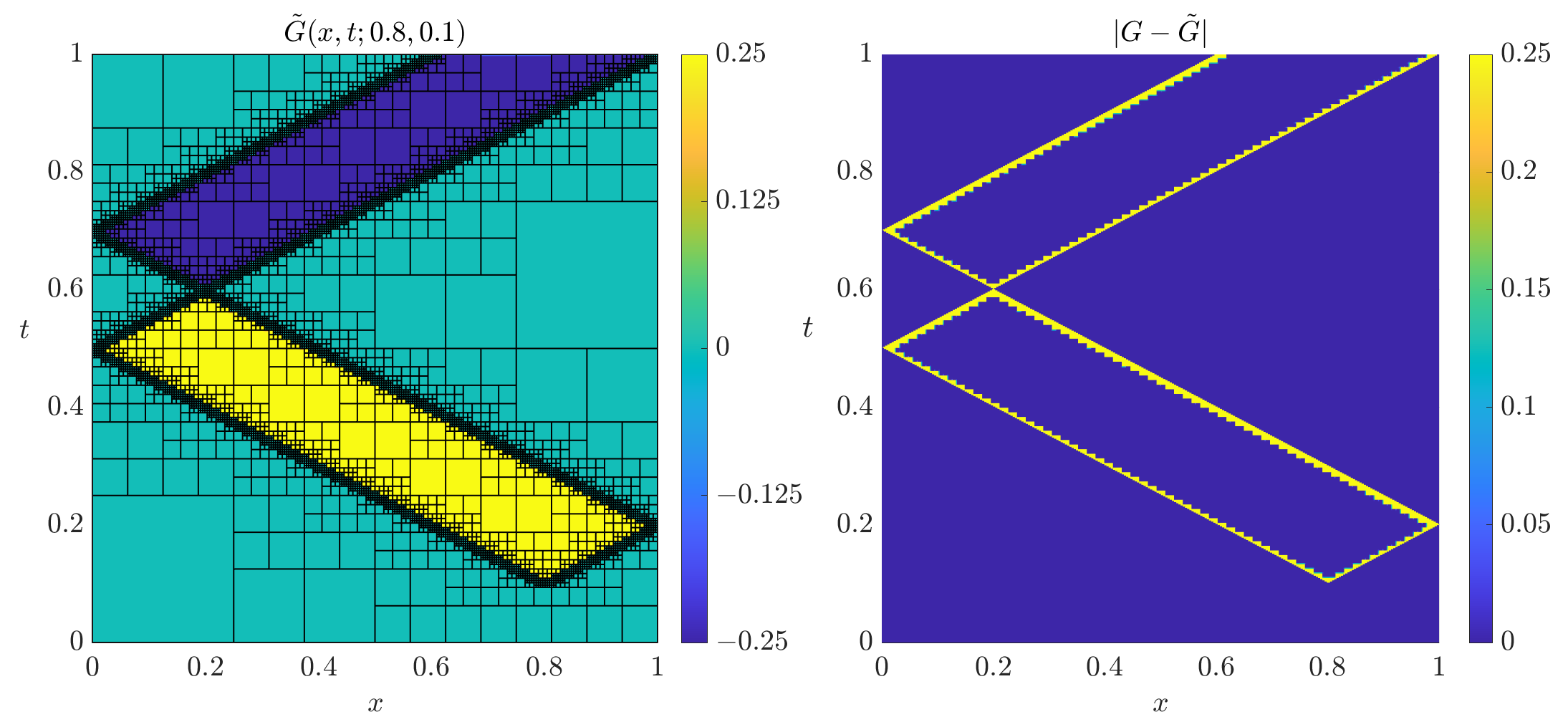}
		\caption{Approximate Green's function at the slice $(y,s)=(0.8,0.1)$ for the wave operator $\L u = u_{tt}-4u_{xx}$, overlaid with partition blocks (left), and the pointwise error at the slice compared with the exact Green's function (right).}
		\label{fig:numerics}
	\end{figure}
	
	\begin{figure}[t]
		\centering
		\includegraphics[width=0.7\textwidth]{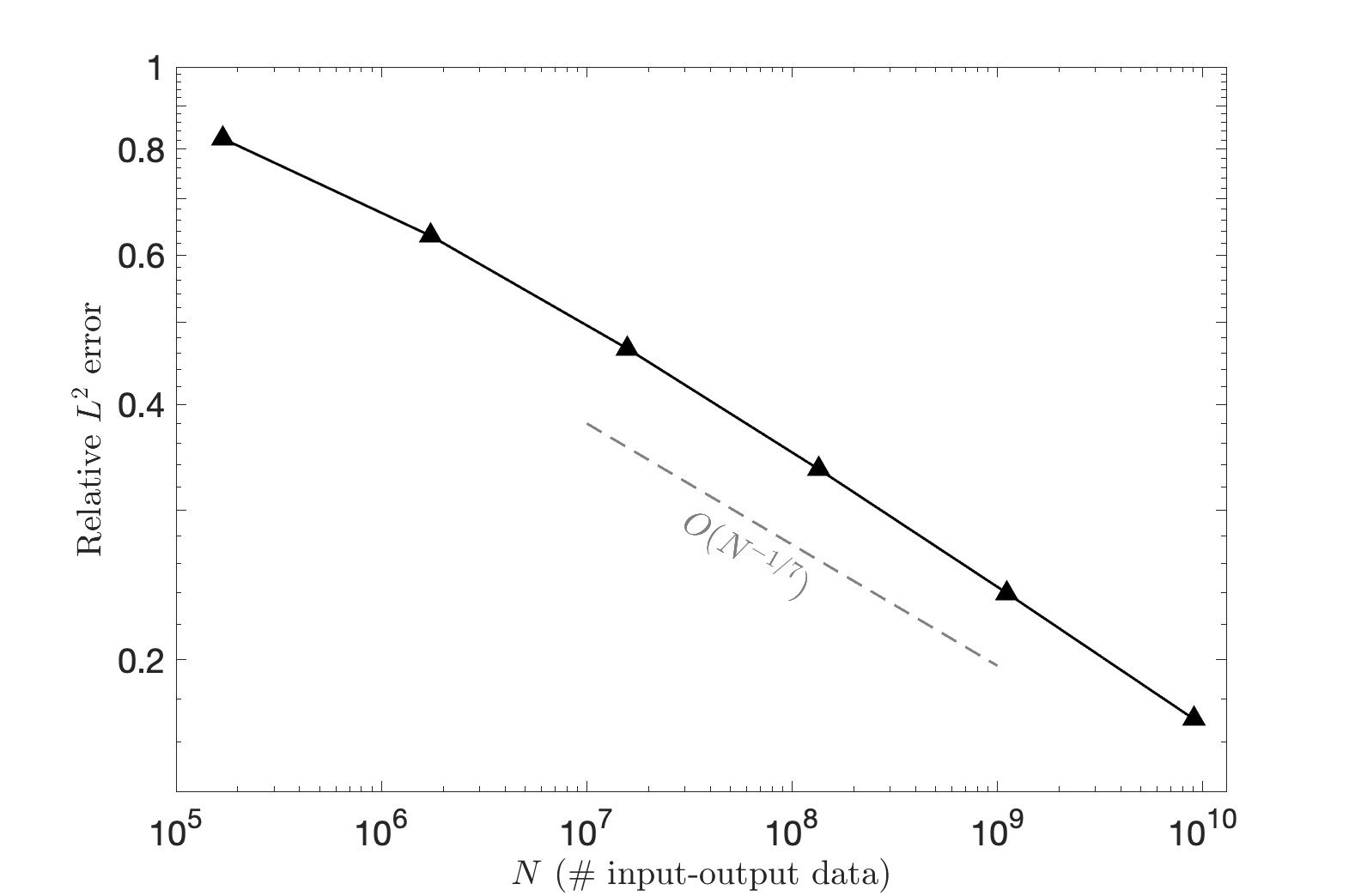}
		\caption{Empirical rate of convergence of Algorithm \ref{alg:main}.}
		\label{fig:converge}
	\end{figure}

	\section{Conclusions and Discussion}\label{s:conclusion}
	
	Our breakthrough result proves that one can rigorously recover the Green's function of a hyperbolic PDE in two variables using input-output training pairs with high probability. We do so via three theoretical contributions: (a) we prove new error bounds for the rSVD for HS operators in the operator norm, such that the suboptimality factor $A_k$ of \eqref{eq:A-simple} tends to a constant as the target rank $k$ tends to $\infty$; (b) we develop a scheme to detect the numerical rank of an operator using input-output products, with high probability; and (c) we introduce an adaptive component to the hierarchical partition of the Green's function's domain to effectively capture the location of its discontinuities.
	
	In the remainder of this section, we discuss possible relaxations of our assumptions, a potential improvement to our scheme via the peeling algorithm, the difficulties of extending to higher-dimensional hyperbolic PDOs, and some challenges arising from the availability of input-output data.
	
	\subsection{Relaxations of assumptions}\label{ss:relaxations}
	
	The main assumptions required of our unknown PDO $\L$ from \eqref{eq:L} are (a) homogeneous initial conditions, (b) self-adjointness, and (c) sufficient regularity of coefficients. Strict hyperbolicity is also required, but we do not discuss it because it enforces realistic physics by forbidding waves from traveling infinitely fast or backward in time.
	
	\subsubsection*{Inhomogeneous initial conditions}
	For use as a direct solver, the homogeneous Green's function can be applied to solve Cauchy problems with certain inhomogeneous initial conditions, as a consequence of Duhamel's 
	principle \citep[Ch.~VI.15]{Courant1962}. For initial conditions $u(x,0)=0$ and $u_t(x,0)=\psi(x)$, one simply needs to integrate $\psi$ against the Green's function on the line $\{t=0\}$. If one also wants to solve initial value problems where $u(x,0)$ is not identically zero, then the values and derivatives of the coefficients of $\L$ must be given on $\{t=0\}$ \citep[Ch.~V.5]{Courant1962}. We demonstrate the practical implementation of Duhamel's principle in Figure \ref{fig:ivp} by replacing the given inhomogeneous initial condition with an appropriate approximation of the Dirac delta on a short time strip, e.g., $f_\delta(x,t) := \delta^{-1}\chi_{[0,\delta]}(t)u_t(x,0)$, and setting our approximate solution to be
	$$\tilde u(x,t) = \int_{D_T}\tilde G(x,t;y,s)f_\delta(y,s)\dd y\dd s,
		\qquad (x,t)\in D_T.$$
	The inaccuracies seen in Figure \ref{fig:ivp} are due primarily to the use of an approximate Green's function. Because of the large amount of training data needed here, we are unable to construct a Green's function with sufficiently high accuracy with our available software and hardware.
	
	\begin{figure}[h]
		\centering
		\includegraphics[width=\textwidth]{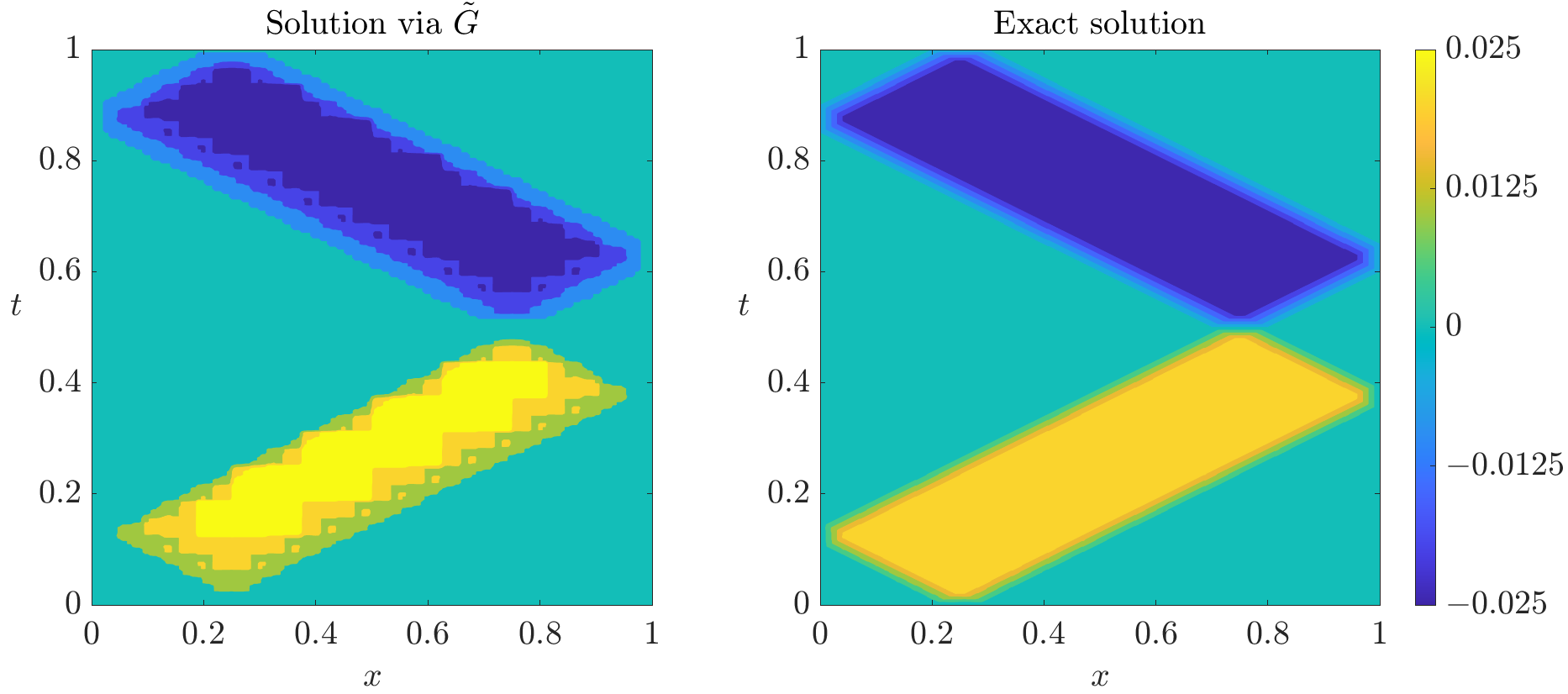}
		\caption{Approximate solution (left) and exact solution (right) to the wave operator $\L u=u_{tt}-4u_{xx}$ with homogeneous boundary conditions but inhomogeneous, discontinuous initial conditions given by $u(x,0)=0$ and $u_t(x,0)=\chi_{[0.2,0.3]}(x)$. The approximation was computed by multiplying the approximate Green's function $\tilde G$ derived from Algorithm \ref{alg:main} against an approximation of the Dirac delta on $\{t=0\}$.}
		\label{fig:ivp}
	\end{figure}
	
	One may be tempted to extrapolate the learned solution operator beyond the given time horizon by taking the value of the computed solution at the end of the first time interval to be the initial condition for the next time interval. However, without additional information about the coefficients of $\L$, one has no information whatsoever about the behavior of the Green's function beyond the time horizon on which the input-output training data is given.  If it is known that the wave speed, for instance, is independent of time, then such extrapolations beyond the initial domain may be feasible.

	\subsubsection*{Non-self-adjoint PDOs}
	While self-adjointness is not strictly necessary, more training data may be needed to learn the solution operator of a non-self-adjoint PDO to the same accuracy. This is because the coefficients of the adjoint PDO may have less regularity than the original operator. For instance, the wave operator with variable damping
	$$\L u = u_{tt} - u_{xx} + k(x,t)u_t$$
	has adjoint
	$$\L^*v = v_{tt} - v_{xx} - k(x,t)v_t - k_t(x,t)v.$$
	If $k(x,t)$ is not smooth, then one of the coefficients of $\L^*$ has one less degree of regularity. Consequently, the corresponding homogeneous Green's function also has one less degree of regularity in the region of its support and thus requires more training data to recover (see Section \ref{ss:low-rank}).
	
	Whether or not our hyperbolic PDO is self-adjoint, in general one must have access to input-output data from both the Cauchy problem \eqref{eq:problem} as well as the corresponding adjoint problem (see footnote \ref{fn:adjoint}). Indeed, one cannot expect to recover an operator without the action of the adjoint \citep{Boulle2024, Halikias2024, Otto2023}. In the special case where the hyperbolic PDO is self-adjoint and has time-independent coefficients, the homogeneous Cauchy problem is also self-adjoint.
	
	\subsubsection*{Less regular coefficients}
	We assume throughout the paper that the coefficients of $\L$ are at least of class $\cC^1$. It is certainly possible that our method extends to coefficients of lower regularity, but the Green's function theory for hyperbolic PDEs with low regularity coefficients is more subtle and may require one to work instead with parametrices. For example, \citet{Smith1998} describes a parametrix for hyperbolic PDEs with Lipschitz coefficients whose spatial derivatives are also Lipschitz. For equations with even lower regularity, it is not always the case that the Cauchy problem is well-posed, and the regularity in the time and the space variables must be considered separately \citep[see, e.g.,][]{Cicognani1999, Cicognani2018, Hurd1968, Reissig2003}. It is unclear in such cases whether or not a Green's function approach, such as the one taken in this paper, is possible.
	
	\subsection{Peeling algorithm}\label{ss:peeling}
	
	Our algorithm requires performing rank detection and the rSVD on $O(\epsilon^{-6})$ subdomains. These operations require $O(\Psi_\epsilon^{-1}\epsilon^{-1}\log(\Xi_\epsilon^{-1}\epsilon^{-1}))$ input-output training pairs for each subdomain; considering the subdomains one by one leads us to the number of training pairs described in Theorem \ref{thm:main}. 
	
	However, it was recently shown in \citet{Boulle2023} that one can dramatically reduce the required number of input-output pairs by considering multiple subdomains simultaneously via the peeling algorithm \citep{Lin2011}. The argument relies in part on bounding the chromatic number of a graph associated with the hierarchical partition \citep{Levitt2022}. In our setting, the chromatic number at a given level $L$ of the partition, denoted by $\chi(L)$, is approximately the number of subdomains of minimal volume---that is, cubes with side length $2^{-L}$---seen in a 2-dimensional slice of the domain $D_T\times D_T$. For instance, the chromatic number associated with the partition at level $L=3$ seen in Figure \ref{fig:hierarchy-2} is 52. Heuristically, the chromatic number is maximized when the characteristics are straight lines of some minimal slope $0<m\le1$, in which case $\chi(L) \approx 2^Lm^{-1}$. If the arguments in \citet{Boulle2023} can be repeated for hyperbolic PDOs, then the required number of input-output training pairs in Theorem \ref{thm:main} improves from $O(\Psi_\epsilon^{-1}\epsilon^{-7}\log(\Xi_\epsilon^{-1}\epsilon^{-1}))$ to $O(\Psi_\epsilon^{-1}\epsilon^{-3}\polylog(\Xi_\epsilon^{-1}\epsilon^{-1}))$. Moreover, by Remark \ref{rmk:regularity}, this term continues to improve as one assumes greater regularity of the coefficients $a,b,c$ of \eqref{eq:L}. The challenge, which is addressed in \citet{Boulle2023} for elliptic PDEs, is that the peeling algorithm of \citet{Lin2011} is not proven to be stable. Hence, careful analysis is needed to show that errors from low-rank approximation and rank detection do not accumulate.
	
	Notice that, unlike the elliptic case, the use of peeling cannot entirely reduce the $\epsilon^{-6}$ factor to $\polylog(\epsilon^{-1})$. This is because the singularities of the Green's functions of hyperbolic PDEs in $d$ variables occupy a $(2d-1)$-dimensional hypersurface in the $2d$-dimensional domain, due to the characteristics. In contrast, the singularities of the Green's functions of elliptic PDEs lie only on the diagonal, which is $d$-dimensional. In other words, the Green's functions of hyperbolic PDEs have an additional $d-1$ dimensions worth of singularities, and the chromatic number $\chi(L)$ thus grows exponentially with $L$. It is therefore doubtful that one can recover the Green's functions of hyperbolic PDEs at a sublinear rate without a significant development in the approximation theory of hierarchical matrices.

	\subsection{Extending to higher dimensions}
	
	We learn the Green's function of a hyperbolic PDE in two variables by detecting the location of its characteristics, which feature as jump discontinuities for the Green's function. A similar scheme should also work for detecting the singularities of Green's functions for hyperbolic PDEs in higher dimensions.
	
	The main difficulty of higher-dimensional Green's functions is that they are not necessarily square-integrable. For instance, the Green's function of the wave equation in three spatial dimensions is given by
	$$G(x,t;y,s) = \frac{\delta(t-s-\|x-y\|)}{4\pi\|x-y\|},
		\qquad (x,t),(y,s)\in\R^3\times[0,\infty),$$
	where $\delta$ is the Dirac delta. Here, $G$ has singularities of a distributional nature on the characteristics, which form a 7-dimensional hypersurface in the 8-dimensional domain of the Green's function. In this case, the theory of the rSVD for HS operators no longer applies since HS integral operators must have a square-integrable kernel. This difficulty has been circumvented to a degree in the case of parabolic PDEs, whose Green's functions are integrable but not square-integrable on the diagonal \citep{Boulle2022b}. Nevertheless, it is not immediately clear that the analysis carries over when the singularities are distributional.
	
	Another factor that may play a role in higher dimensions is the effect of Huygens' principle \citep{Courant1962, Evans2010}. For the standard wave equation in $d$ spatial dimensions, Huygens' principle implies that the support of the Green's function lies exclusively on the boundary of the future light cone when $d$ is odd. In contrast, the support is the entire future light cone when $d$ is even. For general hyperbolic PDOs, this effect does not correspond so nicely with the parity of the spatial dimension \citep[see, e.g.,][]{Gunther1991}. Determining the impact of Huygens' principle may be an essential aspect of hyperbolic PDE learning in higher dimensions.

	\subsection{Availability of data}\label{ss:data}
	In practice, one often does not have access to exact solutions of the PDE in question. Instead, input-output training data may be generated from simulations driven by a numerical PDE solver. In such cases, the training data inherits whichever discretization errors the solver itself exhibits. 
	
	Figure \ref{fig:bb} displays the output of our algorithm for training data coming from two different numerical solvers. The first is a first-order accurate upwind method \citep{Banks2012} while the second is the simple second-order centered-in-time and centered-in-space finite difference method; both advance by explicit time-steps on a fixed-resolution uniform grid. Although much more advanced solvers exist, we chose these methods for their speed and simplicity given the large number of numerical solves required by our algorithm. Due to the numerical dissipation of the first-order method, we observe that the approximate Green's function is unable to locate the characteristics precisely, since they get smoothed out. On the other hand, the numerical dispersion exhibited by the second-order method fabricates artificial oscillations both inside and outside the light cones; these oscillations mislead our rank detection scheme into believing the Green's function is high-rank in regions where it is not, resulting in excessive partitioning within the light cones. 
%	A further difficulty is the limitation on the number of levels of partitioning imposed by two factors. First, despite our simple solvers, they still take significantly more time than the blockwise matrix multiplication used in the approximation of Figure \ref{fig:numerics}, and the number of solves required for our algorithm grows exponentially with the number of levels. Second, the fixed resolution...

	As such, the availability and accuracy of input-output data poses a major challenge for operator learning for hyperbolic PDEs that we encourage future research---both theoretical and practical---to address. While elliptic and parabolic PDEs significantly smooth out errors from the training data, those errors are structurally propagated throughout the domain in the hyperbolic case and must be dealt with in a careful manner. One possible approach is to use our algorithm to quickly construct a low-accuracy approximation of the Green's function to get a rough sense of where the characteristics lie. Then one can use this information to update the numerical solver generating the input-output data to better track propagation of singularities and minimize dispersion and dissipation errors. Alternating between using the Green's function approximation to improve the solver and vice versa may be an iterative solution to the problems presented above. 
	
	\begin{figure}[h]
		\centering
		\includegraphics[width=\textwidth]{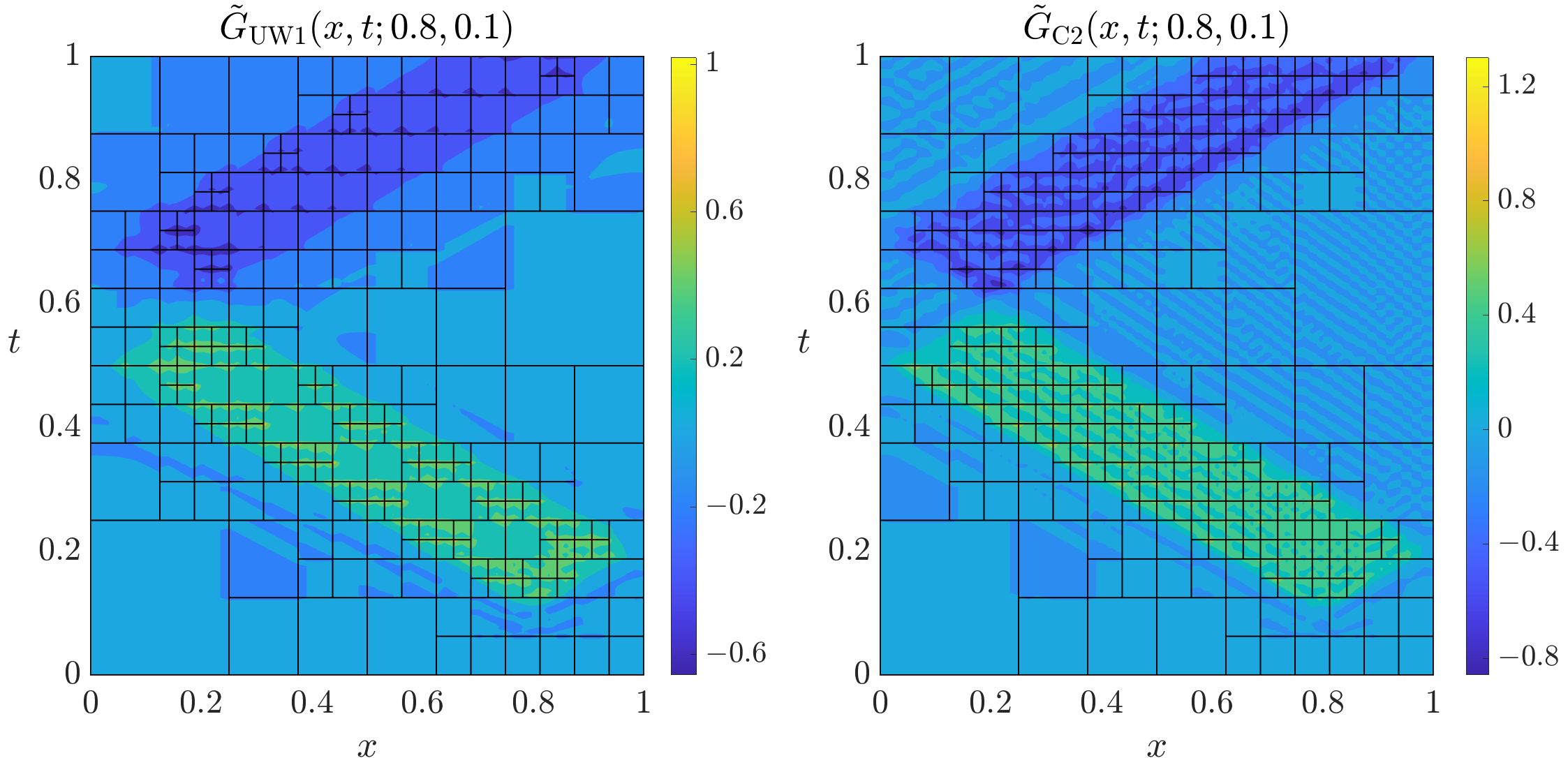}
		\caption{Approximate Green's function at the slice $(y,s)=(0.8,0.1)$ using input-output data generated by time-stepping numerical solvers: the first-order upwind method of \citet{Banks2012} (left) and the second-order centered difference discretization (right).}
		\label{fig:bb}
	\end{figure}

% Acknowledgements and Disclosure of Funding should go at the end, before appendices and references
\acks{The work of C.~W.\ was supported by the NSF GRFP under grant DGE-2139899. The work of A.~T.\ was supported by the Office of Naval Research (ONR) grant N00014-23-1-2729, NSF grant DMS-2045646, and a Weiss Junior Fellowship Award. Many thanks to Diana Halikias, Sungwoo Jeong, and Seick Kim for helpful discussions and suggestions.}

\appendix
\section{Background Material}\label{appendix-A}
	
	We review the relevant background material, comprising of HS operators, quasimatrices, orthogonal projectors, Gaussian processes, and approximation using Legendre series \citep{Halko2011, Hsing2015, Townsend2015, Trefethen2019, Vershynin2018}.
	
	\subsection{Hilbert--Schmidt operators}\label{ss:HS-op}
	
    Hilbert--Schmidt (HS) operators acting on $L^2$ functions are infinite-dimensional analogues of matrices acting on vectors. For $d\ge1$, let $D_1,D_2\subseteq\R^d$ be domains and let $L^2(D_i)$ be the space of square-integrable functions defined on $D_i$, for $i=1,2$, equipped with the inner product $\brak{f,g} := \int_{D_i} f(x)g(x)\dd x$. Since $L^2(D_i)$ are separable Hilbert spaces, then for any compact linear operator $\F:L^2(D_1)\to L^2(D_2)$, there exists by the spectral theorem complete orthonormal bases $\{v_j\}_{j=1}^\infty$ and $\{e_j\}_{j=1}^\infty$ for $L^2(D_1)$ and $L^2(D_2)$ respectively, as well as a sequence of positive numbers $\sigma_1\ge\sigma_2\ge\dots>0$ such that
    \begin{equation}\label{eq:Ff-SVE}
    	\F f = \sum_{\substack{j=1\\\sigma_j>0}}^\infty \sigma_j\brak{v_j,f} e_j, \qquad f\in L^2(D_1),
	\end{equation}
	where equality and convergence hold in the $L^2$ sense. This representation is called the singular value expansion (SVE) of $\F$; we call $\{e_j\}$ and $\{v_j\}$ the left and right singular functions of $\F$, respectively, and we call $\{\sigma_j\}$ the singular values of $\F$. The subspaces spanned by the first $k$ left or right singular functions, for some positive integer $k$, are called dominant left and right $k$-singular subspaces of $\F$, respectively.
	
	We say that $\F$ is a HS operator if its HS norm is finite, i.e.,
    $$\|\F\|_{\HS} := \left(\sum_{j=1}^\infty \|\F v_j\|_{L^2(D_2)}^2\right)^{1/2} < \infty.$$
    Since $\|\F\|_\HS^2 = \sum_{j=1}^\infty\sigma_j^2$, the HS norm is an infinite-dimensional analogue of the Frobenius norm for matrices. We also have $\|\F\|\le\|\F\|_\HS$. A special type of HS operator is a HS integral operator, which requires that $\F$ have an integral representation given by
    \begin{equation}\label{eq:F-kernel}
    	(\F f)(x) = \int_{D_1} G(x,y)f(y)\dd y, \qquad f\in L^2(D_1),\ x\in D_2.
	\end{equation}
    for some kernel $G\in L^2(D_2\times D_1)$. In this case, $\|\F\|_\HS = \|G\|_{L^2(D_2\times D_1)}$. 
    
    The adjoint of $\F$ is the unique compact linear operator $\F^*:L^2(D_2)\to L^2(D_1)$ satisfying $(\F^*)^*=\F$, $\|\F\|=\|\F^*\|$, and $\brak{\F f,h}_{L^2(D_2)}=\brak{f,\F^*h}_{L^2(D_1)}$ for every $f\in L^2(D_1)$ and $h\in L^2(D_2)$. The adjoint $\F^*$ has a SVE given by
    $$\F^*g = \sum_{\substack{j=1\\\sigma_j>0}}^\infty \sigma_j\brak{e_j,g}v_j, 
    	\qquad g\in L^2(D_2),$$
	and for any $f\in L^2(D_1)$ and $h\in L^2(D_2)$ satisfies
	\begin{equation}\label{eq:F*F-FF*}
		(\F^*\F)f = \sum_{j=1}^\infty \sigma_j^2 \brak{v_j,f} v_j, \qquad 
			(\F\F^*)h = \sum_{j=1}^\infty \sigma_j^2 \brak{e_j,h} e_j.
	\end{equation}
    
    By the Eckart--Young--Mirsky theorem \citep[Thm.~4.18]{Stewart1990}, truncating the SVE of a HS operator after $k$ terms yields a best rank-$k$ approximation in the HS and operator norms. If $\F_k:L^2(D_1)\to L^2(D_2)$ is defined by the truncation $\F_kf := \sum_{j=1}^k \sigma_j\brak{v_j,f}e_j$, then
    \begin{equation}\label{eq:eym-HS}
    	\min_{\rank(\tilde\F) = k}\|\F - \tilde\F\|_\HS = \|\F - \F_k\|_\HS 
    		= \left(\sum_{j=k+1}^\infty \sigma_j^2\right)^{1/2}
	\end{equation}
	as well as
	\begin{equation}\label{eq:eym-op}
		\min_{\rank(\tilde\F) = k}\|\F - \tilde\F\| = \|\F - \F_k\| = \sigma_{k+1}.
	\end{equation}
	Finally, we define the numerical rank of $\F$ with error tolerance $\epsilon>0$, denoted by $\rank_\epsilon(\F)$, to be the smallest integer $k$ such that $\|\F - \F_k\| < \epsilon\|\F\|$. In particular, \eqref{eq:eym-op} implies that $\rank_\epsilon(\F)\le k$ holds if $\sigma_{k+1}\le\epsilon\sigma_1$.

    \subsection{Quasimatrices}\label{ss:quasimatrices}
    
    Quasimatrices are infinite-dimensional analogues of tall-skinny matrices, and we use them principally to simplify notation and emphasize the analogy with matrices. A $D_1\times k$ quasimatrix $\bOmega$ is represented by $k$ columns, where each column is a function in $L^2(D_1)$, via
    $$\bOmega = \begin{bmatrix} \omega_1 & \cdots & \omega_k \end{bmatrix},
    	\qquad \omega_j\in L^2(D_1).$$
	One can similarly define $\infty\times k$ infinite matrices, in which each of the $k$ columns is a sequence in $\ell^2$, and likewise $D_1\times\infty$ quasimatrices.
	
	Many operations for rectangular matrices generalize to quasimatrices. If $\F:L^2(D_1)\to L^2(D_2)$ is a HS operator, then $\F\bOmega$ denotes the quasimatrix given by applying $\F$ to each column of $\bOmega$, i.e.,
    $$\F\bOmega = \begin{bmatrix} \F\omega_1 & \cdots & \F\omega_k \end{bmatrix}.$$
   	Quasimatrices have a transpose, denoted $\bOmega^*$, in the sense that
    $$\bOmega^*\bOmega = \begin{bmatrix}
    		\brak{\omega_1,\omega_1} & \cdots & \brak{\omega_1,\omega_k} \\
			\vdots & \ddots & \vdots \\
			\brak{\omega_k,\omega_1} & \cdots & \brak{\omega_k,\omega_k}
		\end{bmatrix}, \qquad
		\bOmega\bOmega^*(x,y) = \sum_{j=1}^k \omega_j(x)\omega_j(y).$$
	Here, $\bOmega^*\bOmega$ is a $k\times k$ matrix of real numbers, while $\bOmega\bOmega^*$ is a function in $L^2(D_1\times D_1)$ and can be thought of as an integral operator by taking $\bOmega\bOmega^*(x,y)$ to be its kernel. Quasimatrices can be thought of as HS operators on the relevant Hilbert spaces if, for instance, their singular values are square summable.

	\subsection{Orthogonal projectors}
	
	In analogy with the finite-dimensional setting, an orthogonal projection is a bounded self-adjoint linear operator $\bP:L^2(D_2)\to L^2(D_2)$ that satisfies $\bP^2=\bP$. Orthogonal projectors are completely determined by their range: for any closed subspace ${\bf M}\subseteq L^2(D_2)$, there is a unique orthogonal projector $\bP_{\bf M}$ whose range is $\bf M$. This operator is compact if and only if $\bf M$ is finite-dimensional. For any $D_2\times k$ quasimatrix $\bOmega$, the unique orthogonal projector associated with the column space of $\bOmega$ is denoted by
	$$\bP_\bOmega:=\bOmega(\bOmega^*\bOmega)^\dagger\bOmega^*:L^2(D_2)\to L^2(D_2),$$
	as $\bP_\bOmega\F:L^2(D_1)\to L^2(D_2)$ captures the orthogonal projection of the range of $\F$ onto the finite-dimensional column space of $\bOmega$. If $\bOmega$ has full column rank, then $\bOmega^*\bOmega$ is an invertible $k\times k$ matrix, and $(\bOmega^*\bOmega)^\dagger=(\bOmega^*\bOmega)^{-1}$. 
	
	We rely on the following inequality, analogous to that of \citet[Prop.~8.6]{Halko2011}, for the operator norm of orthogonal projectors.
	
	\begin{proposition}\label{prop:orthogonal-proj}
		Let $D_1,D_2\subseteq\R^d$ be domains. Let $\bP:L^2(D_2)\to L^2(D_2)$ be an orthogonal projector, and let $\F:L^2(D_1)\to L^2(D_2)$ be a HS operator. Then
		\begin{equation}
			\|\bP\F\| \le \|\bP(\F\F^*)^q\F\|^{1/(2q+1)}
		\end{equation}
		holds for every $q>0$.
	\end{proposition}
	\begin{proof}
		The proof closely follows that of \citet[Prop.~8.6]{Halko2011}, making use of the Courant--Fischer minimax principle \citep[Thm.~4.2.7]{Hsing2015}.
	\end{proof}

	\subsection{Gaussian processes}\label{ss:GP}
	
	A Gaussian process (GP) is an infinite-dimensional analogue of a multivariate Gaussian distribution, and a function drawn from a GP is analogous to a random vector drawn from a Gaussian distribution. For a domain $D\subseteq\R^d$ and a continuous symmetric positive semidefinite kernel $K:D\times D\to\R$, we define a GP with mean $\mu:D\to\R$ and covariance $K$ to be a stochastic process $\{X_t:t\in D\}$ such that the random vector $(X_{t_1},\dots,X_{t_n})$, for every finite set of indices $t_1,\dots,t_n\in D$, is a multivariate Gaussian distribution with mean $(\mu(t_1),\dots,\mu(t_n))$ and covariance matrix $K_{ij}=K(t_i,t_j)$, $1\le i,j\le n$. We denote such a GP by $\GP(\mu,K)$. Functions drawn from a GP with continuous kernel are almost surely in $L^2$. Additionally, since $K$ is positive semidefinite, it has non-negative eigenvalues $\lambda_1\ge\lambda_2\ge \dots \ge0$ with corresponding eigenfunctions $\{r_j\}_{j=1}^\infty$ that form an orthonormal basis for $L^2(D)$ \citep[Thm.~4.6.5]{Hsing2015}. The Karhunen--L\`oeve expansion provides a method for sampling functions from a GP, as $f\sim\GP(0,K)$ has the expansion $f(t) = \sum_{j=1}^\infty \sqrt{\lambda_j} c_j r_j(t)$ converging in $L^2$ uniformly in $t\in D$, where $\{c_j\}_{j=1}^\infty$ are independent standard Gaussian random variables \citep[Thm.~7.3.5]{Hsing2015}. Additionally, we define the trace of $K$ by $\Tr(K) := \sum_{j=1}^\infty\lambda_j < \infty$.

	\subsection{Low-rank approximation in the \texorpdfstring{$L^2$}{L2} norm}\label{ss:low-rank}
	
	In recovering a low-rank approximation of a kernel of a HS integral operator, we need explicit estimates on the decay of the kernel's singular values. We can obtain them by analyzing the kernel expressed in a Legendre series. Let $P_n^*(x)$ denote the shifted Legendre polynomial of degree $n$ in the domain $[0,1]$.\footnote{Legendre polynomials are defined in \citet[Table~18.3.1]{NIST:DLMF}. The shifted Legendre polynomials are simply the composition of Legendre polynomials with the transformation $x\mapsto2(x-1)$ \citep[see][Ch.~18]{NIST:DLMF}.} The shifted Legendre polynomials form an orthogonal basis of $L^2([0,1])$, and if a function $f\in L^2([0,1])$ is uniformly H\"older continuous with parameter $>1/2$, then it has a uniformly convergent Legendre expansion $f(x) = \sum_{n=0}^\infty a_nP_n^*(x)$. The truncated Legendre series $f_k(x) := \sum_{n=0}^k a_nP_n^*(x)$, for some $k\ge0$, is the degree-$k$ $L^2$ projection of $f$ onto the space of polynomials of degree $\leq k$.
	
	If $f$ is $r$-times differentiable and $f^{(r)}$ has bounded total variation $V$, then for any $k>r+1$, we have
	\begin{equation}\label{eq:approx-poly}
		\|f - f_{k-1}\|_{L^2} \le \frac{\sqrt{2}\,V}{\sqrt{\pi(r+1/2)}(k-r-1)^{r+1/2}} = O(k^{-(r+1/2)}),
	\end{equation}
	where $f_{k-1}$ is the degree-$(k-1)$ $L^2$ projection of $f$ \citep[Thm.~3.5]{Wang2023}. Similar results hold for a multivariate function $f(x_1,\dots,x_n)$ if the conditions above hold uniformly over every 1-dimensional slice of $f$, which holds if $f$ is $r$-times differentiable and $D^rf$ has total variation bounded by $V<\infty$ uniformly over all order-$r$ partial derivatives $D^r$ \citep{Shi2021, Townsend2014}. In particular, for any positive integers $m,n$, we may consider $f:[0,1]^{m+n}\to\R$ and its degree-$(k-1)$ $L^2$ projection $f_{k-1}$ as the kernels of HS integral operators $\F,\tilde\F:L^2([0,1]^m)\to L^2([0,1]^n)$, respectively. Since $f_{k-1}$ is a polynomial of degree $k-1$, then $\tilde\F$ is a rank-$k$ operator. By the Eckart--Young--Mirsky theorem,
	$$\left(\sum_{j=k+1}^\infty \sigma_j^2\right)^{1/2} \le \|\F - \tilde\F\|_\HS 
		= \|f - f_{k-1}\|_{L^2} \le \frac{\sqrt{2}\,V}{\sqrt{\pi(r+1/2)}(k-r-1)^{r+1/2}},$$
	where $\sigma_j$ denotes the $j$th singular value of $\F$. Combining this with the observation that $k\sigma_{2k}^2 \le \sum_{j=k+1}^{2k} \sigma_j^2 \le \sum_{j=k+1}^\infty\sigma_j^2$, we obtain an explicit decay rate for the singular values of $\F$, i.e.,
	\begin{equation}\label{eq:approx-sv}
		\sigma_k \le \frac{2^{r+3/2} V}{\sqrt{\pi(r+1/2)} \cdot\sqrt k(k-2r-2)^{r+1/2}} 
			= O(k^{-(r+1)})
	\end{equation}
	for $k > 2(r+1)$.

\section{Randomized Singular Value Decomposition}\label{appendix-B}

This appendix contains proofs of results relating to the rSVD in Section \ref{s:rSVD}.

\subsection{Proofs for Section \ref{ss:rSVD-op}}

In this section, we prove Theorems \ref{thm:rSVD-op} and \ref{thm:rSVD-power}, which bound the error of an approximant generated by the rSVD with respect to the operator norm. The proof is similar to that of \citet[Thms.~10.6 and 10.8]{Halko2011}, generalized to the infinite-dimensional setting with non-standard Gaussian vectors.
	
	We first introduce notation. For separable Hilbert spaces $\cH,\cK$, let $\HS(\cH,\cK)$ be the Hilbert space of HS operators $\cH\to\cK$. Given a bounded, self-adjoint linear operator $\bC:\cK\to\cK$, we define a quadratic form on $\cK$ by 
	$$\brak{x,y}_\bC := \brak{x,\bC y}_\cK, \qquad x,y\in\cK,$$
	as well as a quadratic form on $\HS(\cH,\cK)$ by 
	$$\brak{\bA,\bB}_\bC := \brak{\bA,\bC\bB}_\HS.$$
	If $\bC$ is positive definite, then both quadratic forms are inner products on $\cK$ and $\HS(\cH,\cK)$, respectively. Then we let $\|x\|_\bC^2 := \brak{x,x}_\bC$ and $\|\bA\|_\bC^2 := \brak{\bA,\bA}_\bC$ denote the induced norms. We remark that $\|x\|_\bC=\|\bC^{1/2}x\|_\cK$ and $\|\bA\|_{\bI\to\bC}=\|\bC^{1/2}\bA\|$, where $\|\bA\|_{\bI\to\bC}$ denotes the operator norm of $\bA$ when viewed as an operator $(\cH,\|\cdot\|_\cH)\to(\cK,\|\cdot\|_\bC)$. Finally, we let $\cH\otimes\cK$ denote the tensor product of Hilbert spaces, which embeds into $\HS(\cH,\cK)$ by associating with each $x\otimes y\in\cH\otimes\cK$ the operator $u\mapsto\brak{x,u}_\cH\,y$.
	
	\begin{lemma}\label{lem:op-norm-gaussian-1}
		Let $\cH,\cK$ be Hilbert spaces. For any $u,w\in\cH$ and $v,z\in\cK$, we have
		$$\|u\otimes v - w\otimes z\|_\HS^2 \le \max(\|u\|^2,\|w\|^2)\|v-z\|^2 
			+ \max(\|v\|^2,\|z\|^2)\|u-w\|^2.$$
	\end{lemma}
	\begin{proof}
		Observe that $\brak{u\otimes v,w\otimes z}_\HS = \brak{u,w}\brak{v,z}$. Then
		\begin{align*}
			\|u\otimes v - w\otimes z\|^2 
			&= \|u\otimes v\|^2 + \|w\otimes z\|^2 - 2\brak{u\otimes v,w\otimes z} \\
			&= \|u\|^2\|v\|^2 + \|w\|^2\|z\|^2 - 2\brak{u,w}\brak{v,z}.
		\end{align*}
		Let $A=\max(\|v\|,\|z\|)$ and $B=\max(\|u\|,\|w\|)$. Then
		\begin{align*}
			A^2\|u-w\|^2 &+ B^2\|v-z\|^2 \\
			&= A^2(\|u\|^2 + \|w\|^2) + B^2(\|v\|^2 + \|z\|^2) 
				- 2(A^2\brak{u,w} + B^2\brak{v,z}).
		\end{align*}
		Observe that $(B^2-\brak{u,w})(A^2-\brak{v,z})\ge 0$, hence
		% $$\brak{u,w}\brak{v,z} \ge A^2\brak{u,w} + B^2\brak{v,z} - A^2B^2.$$
		% Thus,
		$$\|u\otimes v - w\otimes z\|_\HS^2 \le \|u\|^2\|v\|^2 + \|w\|^2\|z\|^2
			+ 2A^2B^2 - 2(A^2\brak{u,w} + B^2\brak{v,z}).$$
		It is straightforward to verify that
		$$\|u\|^2\|v\|^2 + \|w\|^2\|z\|^2 + 2A^2B^2 \le A^2(\|u\|^2 + \|w\|^2) 
			+ B^2(\|v\|^2 + \|z\|^2),$$
		which completes the proof.
	\end{proof}
	
	We extend a result of \citet[Prop.~10.1]{Halko2011} to the infinite-dimensional setting, following \citet{Gordon1985, Gordon1988}. To simplify notation, we use $\|\cdot\|$ to denote either the operator, $\ell^2$, or $L^2(D)$ norm depending on context. We also use $\cN(0,\bI\otimes\bC)$ to denote the distribution of a mean zero random matrix with independent columns drawn from $\cN(0,\bC).$
	
	\begin{proposition}\label{prop:gaussian-op-norm-E}
		Let $\bG$ be an $\infty\times n$ infinite random matrix with distribution $\cN(0,\bI_n\otimes\bC)$, where $\bC$ is a symmetric positive definite $\infty\times\infty$ covariance matrix with $\Tr(\bC)<\infty$. Let $\bS$ be a $D\times\infty$ quasimatrix and let $\bT$ be a $n\times\infty$ infinite matrix, both deterministic with $\|\bS\|_\HS,\|\bT\|_\HS<\infty$. Then
		$$\E\|\bS\bG\bT\| \le \Big(\|\bC^{1/2}\|_\HS\|\bT\| + \|\bC^{1/2}\|\|\bT\|_\HS\Big)\|\bS\|.$$
	\end{proposition}
	\begin{proof}
		Let $g\sim\cN(0,\|\bT\|^2\bC)$ be Gaussian in $\ell^2$ and let $h\sim\cN(0,\|\bC^{1/2}\bS\|^2\,\bI_n)$ be a Gaussian random vector in $\R^n$, such that $g$, $h$, and $\bG$ are all independent. Define
		$$X(u,v) := \brak{\bS\bG\bT u,v},\quad Y(u,v) := \brak{\bS g,v} + \brak{\bT^* h,u}$$
		indexed by $u\in\ell^2$ and $v\in L^2(D)$ of unit length. One can show that $X(u,v)$ and $Y(u,v)$ are well-defined, and that $X(\cdot,\cdot)$ and $Y(\cdot,\cdot)$ are mean zero Gaussian processes.
			
		We aim to apply the Sudakov--Fernique inequality \citep[Thm.~7.2.11]{Vershynin2018}.
		% Observe that
		% $$\E[(\bS\bG\bT)_i(x)(\bS\bG\bT)_j(y)]
		% 	= \sum_{\alpha=1}^\infty\sum_{\beta=1}^n S_\alpha(x)S_\alpha(y) T_{\beta_i}T_{\beta j}$$
		% by independence of the entries of $\bG$, and Fubini's theorem. Then, for any unit vectors $u,w\in\ell^2$ and $v,z\in L^2(D)$, we compute
		For any unit-length $u,w\in\ell^2$ and $v,z\in L^2(D)$, set $\bA = (\bT u)\otimes(\bS^*v) - (\bT w)\otimes(\bS^*z)$. We compute
		\begin{align}\label{eq:gaussian-op-norm-E-1}
			\E(X(u,v) - X(w,z))^2
		%	&= \E(\brak{\bS\bG\bT u,v} - \brak{\bS\bG\bT w,z})^2 \\
		%	&= \E\left(\int_D\sum_{i=1}^\infty(\bS\bG\bT)_i(x)(u_iv(x) - w_iz(x))\dd x\right)^2 
		%		\nonumber \\
		%	&= \int_D\int_D\sum_{i=1}^\infty\sum_{j=1}^\infty \sum_{\alpha=1}^\infty\sum_{\beta=1}^n
		%		S_\alpha(x)S_\alpha(y) T_{\beta i}T_{\beta j}
		%		(u_iv(x) - w_iz(x))(u_jv(y) - w_jz(y))\dd x\dd y \\
		%	&= \sum_{\alpha=1}^\infty\sum_{\beta=1}^n \int_D\int_D S_\alpha(x)S_\alpha(y)
		%		((\bT u)_\beta v(x) - (\bT w)_\beta z(x))
		%		((\bT u)_\beta v(y) - (\bT w)_\beta z(y)) \dd x\dd y \\
		%	&= \sum_{\alpha=1}^\infty\sum_{\beta=1}^n 
		%		((\bT u)_\beta(\bS^*v)_\alpha - (\bT w)_\beta(\bS^*z)_\alpha)^2 \nonumber \\
			&= \Tr(\bA^*\bC\bA) \nonumber \\
			&\le \|\bT\|^2\|\bS^*(v-z)\|_\bC^2 + \|\bS\|_{\bI\to\bC}^2\|\bT(u-w)\|^2 \nonumber \\
			&= \|\bT\|^2\|\bC^{1/2}\bS^*(v-z)\|^2+ \|\bC^{1/2}\bS\|^2\|\bT(u-w)\|^2.
		\end{align}
		The first equality follows from fully expanding $(X(u,v)-X(w,z))^2$ and using the definition of $\bG$; the inequality is obtained by applying Lemma \ref{lem:op-norm-gaussian-1} taking $\cH$ to be $\R^n$ with the usual inner product, and $\cK$ to be $\ell^2$ with the inner product induced by $\bC$. 
		
		On the other hand, observe that
		\begin{align*}
			\E\brak{\bS g,v-z}^2
			= \|\bT\|^2(v-z)^*\bS\bC\bS^*(v-z)
		%	&= \E((v-z)^*\bS g((v-z)^*\bS g)^* \\
		%	&= (v-z)^*\bS\E(gg^*)\bS^*(v-z) \\
		%	= \|\bT\|^2(v-z)^*\bS\bS^*(v-z)
			= \|\bT\|^2\|\bC^{1/2}\bS^*(v-z)\|^2,
		\end{align*}
		and similarly $\E\brak{\bT^*h,u-w}^2 = \|\bC^{1/2}\bS\|^2\|\bT(u-w)\|^2$. Therefore, by independence of $g$ and $h$, we have
		\begin{align}\label{eq:gaussian-op-norm-E-2}
			\E(Y(u,v) - Y(w,z))^2
		%	&= \E(\brak{\bS g,v-z} + \brak{\bT^*h,u-w})^2 \\
		%	= \E\brak{\bS g,v-z}^2 + \E\brak{\bT^*h,u-w}^2 \\
			= \|\bT\|^2\|\bC^{1/2}\bS^*(v-z)\|^2 + \|\bC^{1/2}\bS\|^2\|\bT(u-w)\|^2.
		\end{align}
		Combining \eqref{eq:gaussian-op-norm-E-1} and \eqref{eq:gaussian-op-norm-E-2} yields $\E(X(u,v)-X(w,z))^2\le\E(Y(u,v)-Y(w,z))^2$ for every unit-length $u,v,w,z$, so by the Sudakov--Fernique inequality, we have
		\begin{equation}\label{eq:SGT-00}
			\E\sup_{\substack{u\in\ell^2\\\|u\|=1}}\sup_{\substack{v\in L^2(D)\\\|v\|=1}}
			\brak{\bS\bG\bT u,v} \le 
			\E\sup_{\substack{u\in\ell^2\\\|u\|=1}}\sup_{\substack{v\in L^2(D)\\\|v\|=1}}
			(\brak{\bS g,v} + \brak{\bT^*h,u}).
		\end{equation}
		Since $\bS\bG\bT$ is almost surely a bounded operator, then the left-hand side of \eqref{eq:SGT-00} is $\E\|\bS\bG\bT\|$. For the right-hand side, two applications of the Cauchy--Schwarz inequality yield
		$$\E\sup_{\substack{v\in L^2(D)\\\|v\|=1}}\brak{\bS g,v}
			\le (\E\|\bS g\|^2)^{1/2}
			= \Tr(\|\bT\|^2\bS\bC\bS^*)^{1/2}
			\le \|\bC^{1/2}\|_\HS\|\bS\|\|\bT\|$$
		An analogous argument for $\brak{\bT^*h,u}$ combined with \eqref{eq:SGT-00} completes the proof.
	\end{proof}

	Finally, we derive probability estimates for the operator norm of the pseudoinverse of non-standard Gaussian matrices, following \citet{Chen2005} and \citet{Edelman1988}.
	
	\begin{proposition}\label{prop:op-norm-gaussian-pseudo}
		Let $\bG$ be an $m\times n$ random matrix with distribution $\cN(0,\bI_n\otimes\bC)$, where $\bC$ is a symmetric positive definite $m\times m$ covariance matrix with eigenvalues $\lambda_1\ge\dots\ge\lambda_m>0$. If $n\ge m\ge 2$, then for any $t\ge0$, we have
		$$\P\{\|\bG^\dagger\| \ge t\} \le \frac{1}{\sqrt{2\pi(n-m+1)}} 
			\left(\frac{e\sqrt n}{\sqrt{\lambda_m}(n-m+1)}\right)^{n-m+1} t^{-(n-m+1)}.$$
	\end{proposition}
	\begin{proof}
		Let $x_1\ge\dots\ge x_m>0$ be the eigenvalues of the Wishart matrix $\bG\bG^*\sim\cW_m(n,\bC)$. Since $\|\bG^\dagger\| = x_m^{-1/2}$, then we focus on the distribution of $x_m$. It is well-known \citep{James1964, James1968} that the joint density of the eigenvalues is given by
		$$f(x_1,\dots,x_m) 
			= \frac{|\bC|^{-\frac n2}}{K_{m,n}} {}_0F_0(\bC^{-1},-\tfrac12\bX) 
			\prod_{1\le i<j\le m}(x_i - x_j) \prod_{i=1}^m x_i^{\frac{n-m-1}{2}},$$
		where $\bX = \diag(x_1,\dots,x_m)$, $|\cdot|$ is the determinant function, ${}_0F_0$ is the hypergeometric function of two matrix arguments, $\Gamma_m$ is the multivariate Gamma function, and
		$$K_{m,n} = \frac{\pi^{\frac{m^2}{2}}}{2^{\frac{mn}{2}} 
			\Gamma_m(\tfrac n2)\Gamma_m(\tfrac m2)}$$
		\citep[for details and definitions, see][]{Muirhead1982}. We bound the density function since the hypergeometric function is increasing in the eigenvalues of its arguments \citep[Thm.~IV.1]{Kates1981}, that is,
		$${}_0F_0(\bC^{-1},-\tfrac 12\bX) \le {}_0F_0(\tfrac{1}{\lambda_m}\bI_m,-\tfrac12\bX)
			= e^{-\frac{1}{2\lambda_m}\sum_{i=1}^mx_i}$$
		where the equality follows from \citet[Prob.~7.7]{Muirhead1982}. Therefore,
		$$f(x_1,\dots,x_m) \le \frac{|\bC|^{-\frac n2}}{K_{m,n}} 
			e^{-\frac{1}{2\lambda_m}\sum_{i=1}^mx_i}
			\prod_{1\le i<j\le m}(x_i - x_j) \prod_{i=1}^m x_i^{\frac{n-m-1}{2}}.$$
		We bound the density function $f_{x_m}$ of $x_m$ by integration. Let $R_x = \{(x_1,\dots,x_{m-1}) : x_1\ge\dots\ge x_{m-1}\ge x\}$, so
		$$f_{x_m}(x) 
			\le \frac{|\bC|^{-\frac n2}}{K_{m,n}} x^{\frac{n-m-1}{2}} e^{-\frac{x}{2\lambda_m}} 
			\int_{R_x} e^{-\frac{1}{2\lambda_m}\sum_{i=1}^{m-1}x_i}
			\prod_{1\le i<j\le m-1}(x_i-x_j)
			\prod_{i=1}^{m-1}(x_i-x) x_i^{\frac{n-m-1}{2}} \dd x_i.$$
		We bound further by using $x_i-x \le x_i$ and the non-negativity of the integrand, then perform the change of variables $x_i = \lambda_my_i$ to obtain
		$$f_{x_m}(x) \le \frac{|\bC|^{-\frac n2}}{K_{m,n}} \lambda_m^{\frac{(m-1)(n+1)}{2}}
			x^{\frac{n-m-1}{2}} e^{-\frac{x}{2\lambda_m}} 
			\int_{R_0} e^{-\frac12\sum_{i=1}^{m-1}y_i} \prod_{1\le i<j\le m-1}(y_i-y_j)
			\prod_{i=1}^{m-1} y_i^{\frac{n-m+1}{2}} \dd y_i,$$
		where $R_0 = \{(y_1,\dots,y_{m-1}) : y_1\ge\dots\ge y_{m-1}\ge0\}$. Notice that the integrand is simply the unnormalized joint density of the eigenvalues of a $\cW_{m-1}(n+1,\bI_{m-1})$ matrix, so the integral evaluates to $K_{m-1,n+1}$. By \citet[Eq.~(3.14)]{Chen2005}, we have
		$$\frac{K_{m-1,n+1}}{K_{m,n}} 
			= \frac{2^{\frac{n-m-1}{2}}\Gamma(\frac{n+1}{2})}{\Gamma(\frac m2)\Gamma(n-m+1)},$$
		as well as
		$$|\bC|^{-\frac n2}\lambda_m^{\frac{(m-1)(n+1)}{2}} \le \lambda_m^{\frac{m-n-1}{2}}.$$
		Thus,
		$$f_{x_m}(x) \le \frac{2^{\frac{n-m-1}{2}} \lambda_m^{\frac{m-n-1}{2}} \Gamma(\frac{n+1}{2})}
			{\Gamma(\frac m2)\Gamma(n-m+1)} x^{\frac{n-m-1}{2}} e^{-\frac{x}{2\lambda_m}},$$
		and we obtain
		$$\P\{x_m \le x^{-2}\} \le \frac{1}{\Gamma(n-m+2)}
			\left(\frac{\sqrt n}{x \sqrt{\lambda_m}}\right)^{n-m+1}$$
		by following \citet[Lem.~4.1]{Chen2005}. We conclude by applying Stirling's approximation.
	\end{proof}
	
	\begin{corollary}\label{cor:op-norm-gaussian-pseudo-E}
		Under the same assumptions as Proposition \ref{prop:op-norm-gaussian-pseudo}, we have
		$$\E\|\bG^\dagger\| < \frac{e\sqrt n}{\sqrt{\lambda_m}(n-m)}.$$
	\end{corollary}
	\begin{proof}
		The proof is the same as that in \citet[Prop.~A.4]{Halko2011}.
	\end{proof}
	
	We are ready to prove Theorem \ref{thm:rSVD-op}.
	
	\begin{proof}(Theorem \ref{thm:rSVD-op})
		Given the notation of Section \ref{ss:rSVD-op}, let $\bOmega_1 = \bV_1^*\bOmega$ and $\bOmega_2 = \bV_2^*\bOmega$. By \cite[Lem.~1]{Boulle2022a}---which holds for any unitarily invariant norm---we have
		$$\E\|\F - \bP_\bY\F\| \le \|\bSigma_2\| + \E\|\bSigma_2\bOmega_2\bOmega_1^\dagger\|.$$
		To compute $\E\|\bSigma_2\bOmega_2\bOmega_1^\dagger\|$, notice that $\bV^*\bOmega\sim\cN(0,\bI_{k+p}\otimes\bC)$, by \cite[Lem.~1]{Boulle2022a}, with columns that are almost surely in $\ell^2$ since $\Tr(\bC)=\Tr(K)<\infty$. Let $\bOmega_2|\bOmega_1$ denote the random matrix $\bOmega_2$ conditioned on $\bOmega_1$, which, by \citet[Thm.~2 and Cor.~2]{Mandelbaum1984}, has distribution $\cN(\bC_{21}\bC_{11}^{-1}\bOmega_1,\bI_{k+p}\otimes(\bC_{22} - \bC_{21}\bC_{11}^{-1}\bC_{12}))$. Let $\bar\bOmega_2$ be normally distributed with mean zero and the same covariance as $\bOmega_2|\bOmega_1$, sampled independently of $\bOmega_1$ and $\bOmega_2$, so that $\bOmega_2|\bOmega_1\sim\bar\bOmega_2 + \bC_{21}\bC_{11}^{-1}\bOmega_1$. We now condition on $\bOmega_1$ to obtain, by Proposition \ref{prop:gaussian-op-norm-E},
		\begin{align*}
			\E\|&\bSigma_2\bOmega_2\bOmega_1^\dagger\| \\
		%	 = \E\left(\E\left[\|\bSigma_2\bOmega_2\bOmega_1^\dagger\| 
		%		\mid \bOmega_1\right]\right)
			&= \E\|\bSigma_2(\bar\bOmega_2 + \bC_{21}\bC_{11}^{-1}\bOmega_1)\bOmega_1^\dagger\| \\
			&\le \Big(
				\|(\bC_{22} - \bC_{21}\bC_{11}^{-1}\bC_{12})^{1/2}\|_\HS\,\E\|\bOmega_1^\dagger\| 
				+ \|\bC_{22} - \bC_{21}\bC_{11}^{-1}\bC_{12}\|^{1/2}\,\E\|\bOmega_1^\dagger\|_\Fr
				\Big) \|\bSigma_2\| \\
			&\qquad + \E\|\bSigma_2\bC_{21}\bC_{11}^{-1}\bOmega_1\bOmega_1^\dagger\|.
		\end{align*}
		First, observe that $\bOmega_1$ has full rank with probability 1, so $\bOmega_1\bOmega_1^\dagger=\bI_k$. Thus,
		$$\E\|\bSigma_2\bC_{21}\bC_{11}^{-1}\bOmega_1\bOmega_1^\dagger\| 
			\le \lambda_1\|\bSigma_2\|\|\bC_{11}^{-1}\|,$$
		using the bounds $\|\bC_{21}\|\le\|\bC\|\le\lambda_1$. Additionally, since $\bC\succeq\bC_{22}\succeq\bC_{22} - \bC_{21}\bC_{11}^{-1}\bC_{12}$, where $\succeq$ denotes the L\"owner order, then we have
		$$\|(\bC_{22} - \bC_{21}\bC_{11}^{-1}\bC_{12})^{1/2}\|_\HS \le \|\bC^{1/2}\|_\HS
			= \sqrt{\Tr(\bC)} = \sqrt{\Tr(K)}$$
		and $\|\bC_{22} - \bC_{21}\bC_{11}^{-1}\bC_{12}\|^{1/2} \le \sqrt{\lambda_1}$.
		Finally, by Corollary \ref{cor:op-norm-gaussian-pseudo-E} and \citet[Eq.~(10)]{Boulle2022a}, we have
		$$\E\|\bOmega_1^\dagger\| \le \frac{e\sqrt{\|\bC_{11}^{-1}\|(k+p)}}{p},\qquad
			\E\|\bOmega_1^\dagger\|_\Fr = \sqrt{\frac{\Tr(\bC_{11}^{-1})}{p+1}},$$
		which altogether yields
		$$\E\|\bSigma_2\bOmega_2\bOmega_1^\dagger\| 
			\le \left(\frac{e\sqrt{\Tr(K)\|\bC_{11}^{-1}\|(k+p)}}{p}
			+ \sqrt{\frac{\lambda_1\Tr(\bC_{11}^{-1})}{p+1}} 
			+ \lambda_1\|\bC_{11}^{-1}\|\right) \|\bSigma_2\|.$$
		The estimate \eqref{eq:rSVD-op-E} for the average error in the operator norm now follows.
		
		For \eqref{eq:rSVD-op-P}, we follow \citet[Thm.~10.8]{Halko2011} and define, for each $t\ge1$, the event
		$$E_t := \left\{
			\|\bOmega_1^\dagger\| \le \frac{e\sqrt{\|\bC_{11}^{-1}\|(k+p)}}{p+1} \cdot t 
			\quad\text{and}\quad
			\|\bOmega_1^\dagger\|_\Fr \le \sqrt{\frac{\Tr(\bC_{11}^{-1})}{p+1}} \cdot t\right\}$$
		such that $\P(E_t^c) \le 2t^{-p}$ by Proposition \ref{prop:op-norm-gaussian-pseudo} and \citet[Lem.~3]{Boulle2022a}. We also consider the random function $h(\bX):=\|\bSigma_2\bX\bOmega_1^\dagger\|$ defined on the space of $\infty\times(k+p)$ quasimatrices. The corresponding Cameron--Martin space $\cM$ with respect to the distribution $\cN(0,\bI_{k+p}\otimes\bC)$ is the space of $\infty\times(k+p)$ quasimatrices $\bY$ satisfying $\Tr(\bY^*\bC^{-1}\bY)<\infty$, or equivalently, 
		$$\cM = \{\bC^{1/2}\bX : \text{$\bX$ is an $\infty\times(k+p)$ quasimatrix}\},$$
		equipped with the inner product $\brak{\bY,\bZ}_\cM := \Tr(\bY^*\bC^{-1}\bZ)$ \citep[Ch.~2]{Bogachev1998}. Notice that $\|\bY\|_\HS\le\sqrt{\lambda_1}\|\bY\|_\cM$ for every $\bY\in\cM$. Thus, the function $h$ conditioned on $\bOmega_1$ is $\cM$-Lipschitzian with constant $\sqrt{\lambda_1}\|\bSigma_2\|\|\bOmega_1^\dagger\|$, since
		$$|h(\bX+\bY) - h(\bX)| \le \|\bSigma_2\bY\bOmega_1^\dagger\| 
			\le \sqrt{\lambda_1}\|\bSigma_2\|\|\bOmega_1^\dagger\|\|\bY\|_\cM$$
		holds for every quasimatrix $\bX$ and every $\bY\in\cM$. We now apply the concentration inequality of \citet[Thm.~4.5.7]{Bogachev1998} to $h$ conditioned on $\bOmega_1$, to obtain, for every $s\ge0$,
		\begin{multline*}
			\P\left\{\|\bSigma_2\bOmega_2\bOmega_1^\dagger\| 
			> \left(\sqrt{\Tr(K)}\|\bOmega_1^\dagger\| + \sqrt{\lambda_1}\|\bOmega_1^\dagger\|_\Fr 
				+ \lambda_1\|\bC_{11}^{-1}\| + \sqrt{\lambda_1}\|\bOmega_1^\dagger\|\cdot s\right) 
				\|\bSigma_2\|\mid E_t \right\} \\
			\le e^{-s^2/2},
		\end{multline*}
		where we make use of the fact
		$$\E[h(\bOmega_2)\mid\bOmega_1] \le \left(\sqrt{\Tr(K)}\|\bOmega_1^\dagger\| 
			+ \sqrt{\lambda_1}\|\bOmega_1^\dagger\|_\Fr + \lambda_1\|\bC_{11}^{-1}\|\right) 
			\|\bSigma_2\|$$
		by Proposition \ref{prop:gaussian-op-norm-E}. By definition of $E_t$, it follows that
		\begin{multline*}
			\P\left\{\|\bSigma_2\bOmega_2\bOmega_1^\dagger\|
			> \left[\left(\sqrt{\Tr(K)} + s\sqrt{\lambda_1}\right)
				\frac{e\sqrt{\|\bC_{11}^{-1}\|(k+p)}}{p+1} \cdot t\ + \right.\right. \\
			\left.\left. +\ \sqrt{\frac{\lambda_1\Tr(\bC_{11}^{-1})}{p+1}} \cdot t
				+ \lambda_1\|\bC_{11}^{-1}\|\right] \|\bSigma_2\| \mid E_t \right\}
				\le e^{-s^2/2}.
		\end{multline*}
		Since $\P(E_t^c)\le 2t^{-p}$, then we obtain
		\begin{multline*}
			\P\left\{\|\bSigma_2\bOmega_2\bOmega_1^\dagger\|
			> \left[\left(\sqrt{\Tr(K)} + s\sqrt{\lambda_1}\right)
				\frac{e\sqrt{\|\bC_{11}^{-1}\|(k+p)}}{p+1} \cdot t\ + \right.\right. \\
			\left.\left. +\ \sqrt{\frac{\lambda_1\Tr(\bC_{11}^{-1})}{p+1}} \cdot t
				+ \lambda_1\|\bC_{11}^{-1}\|\right] \|\bSigma_2\| \right\}
				\le 2t^{-p} + e^{-s^2/2}.
		\end{multline*}
		We combine this estimate with \citet[Lem.~1]{Boulle2022a} to conclude.
	\end{proof}

	The proof of Theorem \ref{thm:rSVD-power} follows easily.
	\begin{proof}(Theorem \ref{thm:rSVD-power})
		We first claim that $\|\F - \bP_{\bZ}\F\| \le \|\sH - \bP_{\bZ}\sH\|^{1/(2q+1)}$. Indeed, let $\bI$ be the identity operator on $L^2(D_2)$ and observe that $\mathbf I-\bP_{\bZ}$ is an orthogonal projector. Then Proposition \ref{prop:orthogonal-proj} yields
		$$\|(\mathbf I - \bP_{\bZ})\F\| 
			\le \|(\mathbf I - \bP_{\bZ})(\F\F^*)^q\F\|^{1/(2q+1)}
			= \|(\mathbf I - \bP_{\bZ})\sH\|^{1/(2q+1)}.$$
		Now noting that the singular values of $\sH$ are $\sigma_1^{2q+1}\ge\sigma_2^{2q+1}\ge\cdots$, the result immediately follows from Theorem \ref{thm:rSVD-op}.
	\end{proof}
	
\subsection{Proofs for Section \ref{ss:rSVA}}

Here, prove Theorem \ref{thm:rSVA}. In the following, let $\bTheta(\cX,\cY)$ denote the diagonal matrix of canonical angles between two subspaces $\cX,\cY$ of $L^2(D_2)$ with equal finite dimension \citep[Def.~I.5.3]{Stewart1990}. Additionally, for any (quasi)matrix $\bX$, we denote its column space by the calligraphic version of the same symbol, namely, $\cX$.
	
	\begin{lemma}\label{lem:rSVA-0}
		Let $\F:L^2(D_1)\to L^2(D_2)$ be a HS operator with SVE given by \eqref{eq:Ff-SVE}. Select an integer $k\ge1$ such that $\sigma_k > \sigma_{k+1}$. Select an integer $q\ge0$ and set
		\begin{equation}\label{eq:rSVA-lem-q}
			\delta_q = \left[\left(\frac{\sigma_k}{\sigma_{k+1}}\right)^{2q+1} - 1\right]^{-1} > 0.
		\end{equation}
		Let $\sH := (\F\F^*)^q\F$ and let $\tilde\sH:L^2(D_1)\to L^2(D_2)$ be a HS operator with finite rank $\ge k$ such that $\|\sH - \tilde\sH\| \le C\sigma_{k+1}^{2q+1}$ for some $C > 0$. Let $\cU_k$ be the dominant left $k$-singular subspace of $\sH$, and let $\tilde\cU_k$ be a dominant left $k$-singular subspace of $\tilde\sH$. If $C\delta_q < 1$, then
		\begin{equation}\label{eq:rSVA-lem}
			\|\sin\bTheta(\cU_k,\tilde\cU_k)\| \le \frac{C\delta_q}{1 - C\delta_q}.
			%A_{k,p} := \left(1 + t^2s^2\frac{3k(k+p)}{\gamma_k(p+1)} 
			%	\sum_{j=1}^\infty\frac{\lambda_j}{\lambda_1}\right)^{1/2}
		\end{equation}
	\end{lemma}
	\begin{proof}
		Let $\tilde\sigma_1\ge\tilde\sigma_2\ge\cdots$ be the singular values of $\tilde\sH$; the singular values of $\sH$ are given by $\sigma_1^{2q+1} \ge \sigma_2^{2q+1} \ge \cdots$. By Weyl's theorem, we have $|\tilde\sigma_k - \sigma_k^{2q+1}| \le C\sigma_{k+1}^{2q+1}$, whereas the assumption \eqref{eq:rSVA-lem-q} implies $\sigma_k^{2q+1} - \sigma_{k+1}^{2q+1} = \delta^{-1}_q\sigma_{k+1}^{2q+1}$. It follows that
		$$\tilde\sigma_k - \sigma_{k+1}^{2q+1} 
			\ge \sigma_k^{2q+1} - |\tilde\sigma_k - \sigma_k^{2q+1}| - \sigma_{k+1}^{2q+1}
			\ge (\delta^{-1}_q - C)\sigma_{k+1}^{2q+1},$$
		which in conjunction with Wedin's theorem \citep[Thm.~V.4.4]{Stewart1990} yields the desired bound.
	\end{proof}
	
	Combining Lemma \ref{lem:rSVA-0} with rSVD means we can approximate an operator's singular values at the extra cost of computing $\tilde\bU_k\F$, which takes an additional $k$ operator-function products. This provides the desired proof. \\
	
	\begin{proof}(Theorem \ref{thm:rSVA})
		Let $\bU_k$ be the $D_1\times k$ quasimatrix whose columns are dominant left $k$-singular functions of $\sH$. A corollary of the CS decomposition \citep[Exc.~I.5.6]{Stewart1990} along with Theorem \ref{thm:rSVD-op} and Lemma \ref{lem:rSVA-0} yields, with high probability,
		$$\|\bU_k - \tilde\bU_k\| \le 2\|\sin\bTheta(\cU_k,\tilde\cU_k)\| 
			\le \frac{2\delta_q A_{k,p}(s,t)}{1 - \delta_q A_{k,p}(s,t)}.$$
		Hence, we find that
		$$\|\bU_k^*\F - \tilde\bU_k^*\F\| 
			\le \frac{2\delta_q A_{k,p}(s,t)}{1 - \delta_q A_{k,p}(s,t)} \|\F\|.$$
		Notice that $\bU_k^*\F$ has the same dominant $k$ singular values as $\F$, so Weyl's theorem completes the proof.
	\end{proof}
	
\vskip 0.2in
\bibliography{bibliography}

\begin{thebibliography}{98}
\providecommand{\natexlab}[1]{#1}
\providecommand{\url}[1]{\texttt{#1}}
\expandafter\ifx\csname urlstyle\endcsname\relax
  \providecommand{\doi}[1]{doi: #1}\else
  \providecommand{\doi}{doi: \begingroup \urlstyle{rm}\Url}\fi

\bibitem[Anderson et~al.(2020)Anderson, Bruno, and Lyon]{Anderson2020}
T.~G. Anderson, O.~P. Bruno, and M.~Lyon.
\newblock High-order, dispersionless ``fast-hybrid'' wave equation solver.
  {P}art {I}: {$O(1)$} sampling cost via incident-field windowing and
  recentering.
\newblock \emph{SIAM J. Sci. Comput.}, 42\penalty0 (2):\penalty0 A1348--A1379,
  2020.

\bibitem[Arora(2023)]{Arora2023}
R.~Arora.
\newblock A deep learning framework for solving hyperbolic partial differential
  equations: Part {I}.
\newblock \emph{\textnormal{Preprint arXiv:2307.04121}}, 2023.

\bibitem[Banks and Henshaw(2012)]{Banks2012}
J.~W. Banks and W.~D. Henshaw.
\newblock Upwind schemes for the wave equation in second-order form.
\newblock \emph{J. Comput. Phys.}, 231\penalty0 (17):\penalty0 5854--5889,
  2012.

\bibitem[Bebendorf and Hackbusch(2003)]{Bebendorf2003}
M.~Bebendorf and W.~Hackbusch.
\newblock Existence of {$\mathcal{H}$}-matrix approximants to the inverse
  {FE}-matrix of elliptic operators with {$L^\infty$}-coefficients.
\newblock \emph{Numer. Math.}, 95:\penalty0 1--28, 2003.

\bibitem[Berman and Peherstorfer(2023)]{Berman2023}
J.~Berman and B.~Peherstorfer.
\newblock Randomized sparse {Neural Galerkin} schemes for solving evolution
  equations with deep networks.
\newblock In \emph{Advances in Neural Information Processing Systems},
  volume~36, pages 4097--4114, 2023.

\bibitem[Berman and Peherstorfer(2024)]{Berman2024}
J.~Berman and B.~Peherstorfer.
\newblock {CoLoRA}: Continuous low-rank adaptation for reduced implicit neural
  modeling of parameterized partial differential equations.
\newblock In \emph{Proceedings of the 41st International Conference on Machine
  Learning}, volume 235, pages 3565--3583, 2024.

\bibitem[Bi et~al.(2023)Bi, Xie, Zhang, Chen, Gu, and Tian]{Bi2023}
K.~Bi, L.~Xie, H.~Zhang, X.~Chen, X.~Gu, and Q.~Tian.
\newblock Accurate medium-range global weather forecasting with {3D} neural
  networks.
\newblock \emph{Nature}, 619\penalty0 (7970):\penalty0 533--538, 2023.

\bibitem[Bogachev(1998)]{Bogachev1998}
V.~I. Bogachev.
\newblock \emph{Gaussian Measures}, volume~62 of \emph{Mathematical Surveys and
  Monographs}.
\newblock American Mathematical Society, 1998.

\bibitem[Boull{\'e} and Townsend(2023)]{Boulle2022a}
N.~Boull{\'e} and A.~Townsend.
\newblock Learning elliptic partial differential equations with randomized
  linear algebra.
\newblock \emph{Found. Comput. Math.}, 23:\penalty0 709--739, 2023.

\bibitem[Boull{\'e} et~al.(2022{\natexlab{a}})Boull{\'e}, Earls, and
  Townsend]{Boulle2022d}
N.~Boull{\'e}, C.~J. Earls, and A.~Townsend.
\newblock Data-driven discovery of {G}reen's functions with
  human-understandable deep learning.
\newblock \emph{Sci. Rep.}, 12:\penalty0 1--9, 2022{\natexlab{a}}.

\bibitem[Boull{\'e} et~al.(2022{\natexlab{b}})Boull{\'e}, Kim, Shi, and
  Townsend]{Boulle2022b}
N.~Boull{\'e}, S.~Kim, T.~Shi, and A.~Townsend.
\newblock Learning {G}reen's functions associated with time-dependent partial
  differential equations.
\newblock \emph{J. Mach. Learn. Res.}, 23\penalty0 (1):\penalty0 1--34,
  2022{\natexlab{b}}.

\bibitem[Boull{\'e} et~al.(2023)Boull{\'e}, Halikias, and Townsend]{Boulle2023}
N.~Boull{\'e}, D.~Halikias, and A.~Townsend.
\newblock Elliptic {PDE} learning is provably data-efficient.
\newblock \emph{Proc. Natl. Acad. Sci. USA}, 120\penalty0 (39):\penalty0
  e2303904120, 2023.

\bibitem[Boull{\'e} et~al.(2024)Boull{\'e}, Halikias, Otto, and
  Townsend]{Boulle2024}
N.~Boull{\'e}, D.~Halikias, S.~E. Otto, and A.~Townsend.
\newblock Operator learning without the adjoint.
\newblock \emph{J. Mach. Learn. Res.}, 25\penalty0 (364):\penalty0 1--54, 2024.

\bibitem[Bruna et~al.(2024)Bruna, Peherstorfer, and Vanden-Eijnden]{Bruna2024}
J.~Bruna, B.~Peherstorfer, and E.~Vanden-Eijnden.
\newblock Neural {Galerkin} schemes with active learning for high-dimensional
  evolution equations.
\newblock \emph{J. Comput. Phys.}, 496:\penalty0 112588, 2024.

\bibitem[Brunton et~al.(2016)Brunton, Proctor, and Kutz]{Brunton2016}
S.~L. Brunton, J.~L. Proctor, and J.~N. Kutz.
\newblock Discovering governing equations from data by sparse identification of
  nonlinear dynamical systems.
\newblock \emph{Proc. Natl. Acad. Sci. USA}, 113\penalty0 (15):\penalty0
  3932--3937, 2016.

\bibitem[Chen and Gu(2021)]{Chen2021}
C.-T. Chen and G.~X. Gu.
\newblock Learning hidden elasticity with deep neural networks.
\newblock \emph{Proc. Natl. Acad. Sci. USA}, 118\penalty0 (31):\penalty0
  e2102721118, 2021.

\bibitem[Chen et~al.(2024)Chen, Wang, and Yang]{Chen2023}
K.~Chen, C.~Wang, and H.~Yang.
\newblock Deep operator learning lessens the curse of dimensionality for
  {PDEs}.
\newblock In \emph{Transactions on Machine Learning Research}, 2024.

\bibitem[Chen and Dongarra(2005)]{Chen2005}
Z.~Chen and J.~Dongarra.
\newblock Condition numbers of {G}aussian random matrices.
\newblock \emph{SIAM J. Matrix Anal. Appl.}, 26\penalty0 (3):\penalty0
  1389--1404, 2005.

\bibitem[Cicognani(1999)]{Cicognani1999}
M.~Cicognani.
\newblock Strictly hyperbolic equations with non regular coefficients with
  respect to time.
\newblock \emph{Ann. Univ. Ferrara}, 45\penalty0 (1):\penalty0 45--58, 1999.

\bibitem[Cicognani and Lorenz(2018)]{Cicognani2018}
M.~Cicognani and D.~Lorenz.
\newblock Strictly hyperbolic equations with coefficients low-regular in time
  and smooth in space.
\newblock \emph{J. Pseudo-Differ. Oper. Appl.}, 9\penalty0 (3):\penalty0
  643--675, 2018.

\bibitem[Courant and Hilbert(1962)]{Courant1962}
R.~Courant and D.~Hilbert.
\newblock \emph{Methods of Mathematical Physics}, volume~2.
\newblock Interscience Publishers, 1st {E}nglish edition, 1962.

\bibitem[{de Hoop} et~al.(2023){de Hoop}, Kovachki, Nelsen, and
  Stuart]{deHoop2023}
M.~V. {de Hoop}, N.~B. Kovachki, N.~H. Nelsen, and A.~M. Stuart.
\newblock Convergence rates for learning linear operators from noisy data.
\newblock \emph{SIAM/ASA J. Uncertain. Quantif.}, 11\penalty0 (2):\penalty0
  480--513, 2023.

\bibitem[Driscoll et~al.(2014)Driscoll, Hale, and Trefethen]{chebfun}
T.~A. Driscoll, N.~Hale, and L.~N. Trefethen, editors.
\newblock \emph{Chebfun Guide}.
\newblock Pafnuty Publications, 2014.

\bibitem[Edelman(1988)]{Edelman1988}
A.~Edelman.
\newblock Eigenvalues and condition numbers of random matrices.
\newblock \emph{SIAM J. Matrix Anal. Appl.}, 9\penalty0 (4):\penalty0 543--560,
  1988.

\bibitem[Evans(2010)]{Evans2010}
L.~C. Evans.
\newblock \emph{Partial Differential Equations}, volume~19 of \emph{Graduate
  Studies in Mathematics}.
\newblock American Mathematical Society, 2nd edition, 2010.

\bibitem[Fan and Ying(2020)]{Fan2020}
Y.~Fan and L.~Ying.
\newblock Solving electrical impedance tomography with deep learning.
\newblock \emph{J. Comput. Phys.}, 404:\penalty0 109119, 2020.

\bibitem[Fan et~al.(2019)Fan, Lin, Ying, and Zepeda-N{\'u}{\~n}ez]{Fan2019}
Y.~Fan, L.~Lin, L.~Ying, and L.~Zepeda-N{\'u}{\~n}ez.
\newblock A multiscale neural network based on hierarchical matrices.
\newblock \emph{Multiscale Model. Simul.}, 17\penalty0 (4):\penalty0
  1189--1213, 2019.

\bibitem[Federer(1959)]{Federer1959}
H.~Federer.
\newblock Curvature measures.
\newblock \emph{Trans. Am. Math. Soc.}, 93\penalty0 (3):\penalty0 418--491,
  1959.

\bibitem[Feliu-Faba et~al.(2020)Feliu-Faba, Fan, and Ying]{Feliu2020}
J.~Feliu-Faba, Y.~Fan, and L.~Ying.
\newblock Meta-learning pseudo-differential operators with deep neural
  networks.
\newblock \emph{J. Comput. Phys.}, 408:\penalty0 109309, 2020.

\bibitem[Gin et~al.(2021)Gin, Shea, Brunton, and Kutz]{Gin2021}
C.~R. Gin, D.~E. Shea, S.~L. Brunton, and J.~N. Kutz.
\newblock {DeepGreen}: Deep learning of {G}reen's functions for nonlinear
  boundary value problems.
\newblock \emph{Sci. Rep.}, 11:\penalty0 1--14, 2021.

\bibitem[Gordon(1985)]{Gordon1985}
Y.~Gordon.
\newblock Some inequalities for {G}aussian processes and applications.
\newblock \emph{Isr. J. Math.}, 50:\penalty0 265--289, 1985.

\bibitem[Gordon(1988)]{Gordon1988}
Y.~Gordon.
\newblock Gaussian processes and almost spherical sections of convex bodies.
\newblock \emph{Ann. Probab.}, 16\penalty0 (1):\penalty0 180--188, 1988.

\bibitem[Gray(2003)]{Gray2003}
A.~Gray.
\newblock \emph{Tubes}, volume 221 of \emph{Progress in Mathematics}.
\newblock Birkh{\"a}user, 2nd edition, 2003.

\bibitem[G{\"u}nther(1991)]{Gunther1991}
P.~G{\"u}nther.
\newblock Huygens' principle and {H}adamard's conjecture.
\newblock \emph{Math. Intell.}, 13\penalty0 (2):\penalty0 56--63, 1991.

\bibitem[Guo et~al.(2020)Guo, Cao, Liu, and Gao]{Guo2020}
Y.~Guo, X.~Cao, B.~Liu, and M.~Gao.
\newblock Solving partial differential equations using deep learning and
  physical constraints.
\newblock \emph{Appl. Sci.}, 10\penalty0 (17):\penalty0 5917, 2020.

\bibitem[Halikias and Townsend(2024)]{Halikias2024}
D.~Halikias and A.~Townsend.
\newblock Structured matrix recovery from matrix-vector products.
\newblock \emph{Numer. Linear Algebra Appl.}, 31\penalty0 (1):\penalty0 e2531,
  2024.

\bibitem[Halko et~al.(2011)Halko, Martinsson, and Tropp]{Halko2011}
N.~Halko, P.-G. Martinsson, and J.~A. Tropp.
\newblock Finding structure with randomness: Probabilistic algorithms for
  constructing approximate matrix decompositions.
\newblock \emph{{SIAM} Rev.}, 53\penalty0 (2):\penalty0 217--288, 2011.

\bibitem[Hsing and Eubank(2015)]{Hsing2015}
T.~Hsing and R.~Eubank.
\newblock \emph{Theoretical Foundations of Functional Data Analysis, with an
  Introduction to Linear Operators}.
\newblock John Wiley \& Sons, 2015.

\bibitem[Huang and Agarwal(2023)]{Huang2023}
A.~J. Huang and S.~Agarwal.
\newblock On the limitations of physics-informed deep learning: Illustrations
  using first-order hyperbolic conservation law-based traffic flow models.
\newblock \emph{IEEE Open J. Intell. Transp. Syst.}, 4:\penalty0 279--293,
  2023.

\bibitem[Huang et~al.(2025)Huang, Nelsen, and Trautner]{Huang2024}
D.~Z. Huang, N.~H. Nelsen, and M.~Trautner.
\newblock An operator learning perspective on parameter-to-observable maps.
\newblock \emph{Found. Data Sci.}, 7\penalty0 (1):\penalty0 163--225, 2025.

\bibitem[Hug et~al.(2004)Hug, Last, and Weil]{Hug2004}
D.~Hug, G.~Last, and W.~Weil.
\newblock A local {Steiner-type} formula for general closed sets and
  applications.
\newblock \emph{Math. Z.}, 246:\penalty0 237--272, 2004.

\bibitem[Hurd and Sattinger(1968)]{Hurd1968}
A.~E. Hurd and D.~H. Sattinger.
\newblock Questions of existence and uniqueness for hyperbolic equations with
  discontinuous coefficients.
\newblock \emph{Trans. Am. Math. Soc.}, 132\penalty0 (1):\penalty0 159--174,
  1968.

\bibitem[James(1964)]{James1964}
A.~T. James.
\newblock Distributions of matrix variates and latent roots derived from normal
  samples.
\newblock \emph{Ann. Math. Stat.}, 35\penalty0 (2):\penalty0 475--501, 1964.

\bibitem[James(1968)]{James1968}
A.~T. James.
\newblock Calculation of zonal polynomial coefficients by use of the
  {Laplace--Beltrami} operator.
\newblock \emph{Ann. Math. Stat.}, 39\penalty0 (5):\penalty0 1711--1718, 1968.

\bibitem[Kaiser et~al.(2018)Kaiser, Kutz, and Brunton]{Kaiser2018}
E.~Kaiser, J.~N. Kutz, and S.~L. Brunton.
\newblock Sparse identification of nonlinear dynamics for model predictive
  control in the low-data limit.
\newblock \emph{Proc. R. Soc. Lond. A}, 474\penalty0 (2219):\penalty0 20180335,
  2018.

\bibitem[Karniadakis et~al.(2021)Karniadakis, Kevrekidis, Lu, Perdikaris, Wang,
  and Yang]{Karniadakis2021}
G.~E. Karniadakis, I.~G. Kevrekidis, L.~Lu, P.~Perdikaris, S.~Wang, and
  L.~Yang.
\newblock Physics-informed machine learning.
\newblock \emph{Nat. Rev. Phys.}, 3\penalty0 (6):\penalty0 422--440, 2021.

\bibitem[Kates(1981)]{Kates1981}
L.~K. Kates.
\newblock \emph{Zonal polynomials}.
\newblock PhD thesis, Princeton University, 1981.

\bibitem[Khoo and Ying(2019)]{Khoo2019}
Y.~Khoo and L.~Ying.
\newblock {SwitchNet}: a neural network model for forward and inverse
  scattering problems.
\newblock \emph{SIAM J. Sci. Comput.}, 41\penalty0 (5):\penalty0 A3182--A3201,
  2019.

\bibitem[Kochkov et~al.(2021)Kochkov, Smith, Alieva, Wang, Brenner, and
  Hoyer]{Kochkov2021}
D.~Kochkov, J.~A. Smith, A.~Alieva, Q.~Wang, M.~P. Brenner, and S.~Hoyer.
\newblock Machine learning-accelerated computational fluid dynamics.
\newblock \emph{Proc. Natl. Acad. Sci. USA}, 118\penalty0 (21):\penalty0
  e2101784118, 2021.

\bibitem[Kovachki et~al.(2023)Kovachki, Li, Liu, Azizzadenesheli, Bhattacharya,
  Stuart, and Anandkumar]{Kovachki2023}
N.~Kovachki, Z.~Li, B.~Liu, K.~Azizzadenesheli, K.~Bhattacharya, A.~Stuart, and
  A.~Anandkumar.
\newblock Neural operator: Learning maps between function spaces with
  applications to {PDEs}.
\newblock \emph{J. Mach. Learn. Res.}, 24:\penalty0 1--97, 2023.

\bibitem[Krishnapriyan et~al.(2021)Krishnapriyan, Gholami, Zhe, Kirby, and
  Mahoney]{Krishnapriyan2021}
A.~S. Krishnapriyan, A.~Gholami, S.~Zhe, R.~M. Kirby, and M.~Mahoney.
\newblock Characterizing possible failure modes in physics-informed neural
  networks.
\newblock In \emph{Advances in Neural Information Processing Systems},
  volume~35, pages 26548--26560, 2021.

\bibitem[Kutz(2017)]{Kutz2017}
J.~N. Kutz.
\newblock Deep learning in fluid dynamics.
\newblock \emph{J. Fluid Mech.}, 814:\penalty0 1--4, 2017.

\bibitem[Lam et~al.(2023)Lam, Sanchez-Gonzalez, Willson, Wirnsberger,
  Fortunato, Alet, Ravuri, Ewalds, Eaton-Rosen, Hu, Merose, Hoyer, Holland,
  Vinyals, Stott, Pritzel, Mohamed, and Battaglia]{Lam2023}
R.~Lam, A.~Sanchez-Gonzalez, M.~Willson, P.~Wirnsberger, M.~Fortunato, F.~Alet,
  S.~Ravuri, T.~Ewalds, Z.~Eaton-Rosen, W.~Hu, A.~Merose, S.~Hoyer, G.~Holland,
  O.~Vinyals, J.~Stott, A.~Pritzel, S.~Mohamed, and P.~Battaglia.
\newblock Learning skillful medium-range global weather forecasting.
\newblock \emph{Science}, 382\penalty0 (6677):\penalty0 1416--1421, 2023.

\bibitem[Lanthaler et~al.(2022)Lanthaler, Mishra, and
  Karniadakis]{Lanthaler2022}
S.~Lanthaler, S.~Mishra, and G.~E. Karniadakis.
\newblock Error estimates for {DeepONets}: A deep learning framework in
  infinite dimensions.
\newblock \emph{Trans. Math. Appl.}, 6\penalty0 (1):\penalty0 tnac001, 2022.

\bibitem[Laurent et~al.(2021)Laurent, Legendre, and Salomon]{Laurent2021}
P.~Laurent, G.~Legendre, and J.~Salomon.
\newblock On the method of reflections.
\newblock \emph{Numer. Math.}, 148\penalty0 (2):\penalty0 449--493, 2021.

\bibitem[Lax(2006)]{Lax2006}
P.~D. Lax.
\newblock \emph{Hyperbolic Partial Differential Equations}, volume~14 of
  \emph{Courant Lecture Notes in Mathematics}.
\newblock American Mathematical Society, 2006.

\bibitem[Lerner(1991)]{Lerner1991}
M.~E. Lerner.
\newblock Qualitative properties of the {R}iemann function [in {R}ussian].
\newblock \emph{Differ. Uravn.}, 27\penalty0 (12):\penalty0 2106--2120, 1991.

\bibitem[Levitt and Martinsson(2024)]{Levitt2022}
J.~Levitt and P.-G. Martinsson.
\newblock Randomized compression of rank-structured matrices accelerated with
  graph coloring.
\newblock \emph{J. Comput. Appl. Math.}, 451:\penalty0 116044, 2024.

\bibitem[Li et~al.(2022)Li, Demanet, and Zepeda-N{\'u}{\~n}ez]{Li2022}
M.~Li, L.~Demanet, and L.~Zepeda-N{\'u}{\~n}ez.
\newblock Wide-band butterfly network: stable and efficient inversion via
  multi-frequency neural networks.
\newblock \emph{Multiscale Model. Simul.}, 20\penalty0 (4):\penalty0
  1191--1227, 2022.

\bibitem[Li et~al.(2020{\natexlab{a}})Li, Kovachki, Azizzadenesheli, Liu,
  Bhattacharya, Stuart, and Anandkumar]{Li2020a}
Z.~Li, N.~Kovachki, K.~Azizzadenesheli, B.~Liu, K.~Bhattacharya, A.~Stuart, and
  A.~Anandkumar.
\newblock Graph kernel network for partial differential equations.
\newblock In \emph{International Conference on Learning Representations
  Workshop on Integration of Deep Neural Models and Differential Equations},
  2020{\natexlab{a}}.

\bibitem[Li et~al.(2020{\natexlab{b}})Li, Kovachki, Azizzadenesheli, Liu,
  Stuart, Bhattacharya, and Anandkumar]{Li2020b}
Z.~Li, N.~Kovachki, K.~Azizzadenesheli, B.~Liu, A.~Stuart, K.~Bhattacharya, and
  A.~Anandkumar.
\newblock Multipole graph neural operator for parametric partial differential
  equations.
\newblock In \emph{Advances in Neural Information Processing Systems},
  volume~33, pages 6755--6766, 2020{\natexlab{b}}.

\bibitem[Li et~al.(2021)Li, Kovachki, Azizzadenesheli, Liu, Bhattacharya,
  Stuart, and Anandkumar]{Li2021}
Z.~Li, N.~Kovachki, K.~Azizzadenesheli, B.~Liu, K.~Bhattacharya, A.~Stuart, and
  A.~Anandkumar.
\newblock Fourier neural operator for parametric partial differential
  equations.
\newblock In \emph{International Conference on Learning Representations}, 2021.

\bibitem[Lin et~al.(2011)Lin, Lu, and Ying]{Lin2011}
L.~Lin, J.~Lu, and L.~Ying.
\newblock Fast construction of hierarchical matrix representation from
  matrix-vector multiplication.
\newblock \emph{J. Comput. Phys.}, 230\penalty0 (10):\penalty0 4071--4087,
  2011.

\bibitem[Liu et~al.(2021)Liu, Xing, Guo, Michielssen, Ghysels, and Li]{Liu2021}
Y.~Liu, X.~Xing, H.~Guo, E.~Michielssen, P.~Ghysels, and X.~S. Li.
\newblock Butterfly factorization via randomized matrix-vector multiplications.
\newblock \emph{SIAM J. Sci. Comput.}, 43\penalty0 (2):\penalty0 A883--A907,
  2021.

\bibitem[Liu et~al.(2023)Liu, Song, Burridge, and Qian]{Liu2023}
Y.~Liu, J.~Song, R.~Burridge, and J.~Qian.
\newblock A fast butterfly-compressed {Hadamard--Babich} integrator for
  high-frequency {H}elmholtz equations in inhomogeneous media with arbitrary
  sources.
\newblock \emph{Multiscale Model. Simul.}, 21\penalty0 (1):\penalty0 269--308,
  2023.

\bibitem[Lu et~al.(2021{\natexlab{a}})Lu, Jin, Pang, Zhang, and
  Karniadakis]{Lu2021a}
L.~Lu, P.~Jin, G.~Pang, Z.~Zhang, and G.~E. Karniadakis.
\newblock Learning nonlinear operators via {DeepONet} based on the universal
  approximation theorem of operators.
\newblock \emph{Nat. Mach. Intell.}, 3:\penalty0 218--229, 2021{\natexlab{a}}.

\bibitem[Lu et~al.(2021{\natexlab{b}})Lu, Meng, Mao, and Karniadakis]{Lu2021b}
L.~Lu, X.~Meng, Z.~Mao, and G.~E. Karniadakis.
\newblock {DeepXDE}: A deep learning library for solving differential
  equations.
\newblock \emph{SIAM Rev.}, 63\penalty0 (1):\penalty0 208--228,
  2021{\natexlab{b}}.

\bibitem[Mackie(1965)]{Mackie1965}
A.~G. Mackie.
\newblock Green's functions and {R}iemann's method.
\newblock \emph{Proc. Edinburgh Math. Soc.}, 14\penalty0 (4):\penalty0
  293--302, 1965.

\bibitem[Mandelbaum(1984)]{Mandelbaum1984}
A.~Mandelbaum.
\newblock Linear estimators and measurable linear transformations on a
  {H}ilbert space.
\newblock \emph{Z. Wahrscheinlichkeitstheorie verw. Gebiete}, 65\penalty0
  (3):\penalty0 385--397, 1984.

\bibitem[Massei et~al.(2022)Massei, Robol, and Kressner]{Massei2022}
S.~Massei, L.~Robol, and D.~Kressner.
\newblock Hierarchical adaptive low-rank format with applications to
  discretized partial differential equations.
\newblock \emph{Numer. Linear Algebra Appl.}, 29\penalty0 (6):\penalty0 e2448,
  2022.

\bibitem[Meier and Nakatsukasa(2024)]{Meier2024}
M.~Meier and Y.~Nakatsukasa.
\newblock Fast randomized numerical rank estimation for numerically low-rank
  matrices.
\newblock \emph{Linear Algebra Appl.}, 686:\penalty0 1--32, 2024.

\bibitem[Molinaro et~al.(2023)Molinaro, Yang, Engquist, and
  Mishra]{Molinaro2023}
R.~Molinaro, Y.~Yang, B.~Engquist, and S.~Mishra.
\newblock Neural inverse operators for solving {PDE} inverse problems.
\newblock In \emph{Proceedings of the 40th International Conference on Machine
  Learning}, volume 202, pages 25105--25139, 2023.

\bibitem[Muirhead(1982)]{Muirhead1982}
R.~J. Muirhead.
\newblock \emph{Aspects of Multivariate Statistical Theory}.
\newblock John Wiley \& Sons, 1982.

\bibitem[NIST DLMF()]{NIST:DLMF}
NIST DLMF.
\newblock {\it NIST Digital Library of Mathematical Functions}.
\newblock Release 1.1.12 of 2023-12-15.
\newblock URL \url{https://dlmf.nist.gov/}.
\newblock F.~W.~J. Olver, A.~B. {Olde Daalhuis}, D.~W. Lozier, B.~I. Schneider,
  R.~F. Boisvert, C.~W. Clark, B.~R. Miller, B.~V. Saunders, H.~S. Cohl, and
  M.~A. McClain, eds.

\bibitem[Otto et~al.(2023)Otto, Padovan, and Rowley]{Otto2023}
S.~E. Otto, A.~Padovan, and C.~W. Rowley.
\newblock Model reduction for nonlinear systems by balanced truncation of state
  and gradient covariance.
\newblock \emph{SIAM J. Sci. Comput.}, 45\penalty0 (5):\penalty0 A2325--A2355,
  2023.

\bibitem[Qian et~al.(2020)Qian, Kramer, Peherstorfer, and Willcox]{Qian2020}
E.~Qian, B.~Kramer, B.~Peherstorfer, and K.~Willcox.
\newblock Lift {\&} learn: Physics-informed machine learning for large-scale
  nonlinear dynamical systems.
\newblock \emph{Physica D}, 406:\penalty0 132401, 2020.

\bibitem[Raissi et~al.(2020)Raissi, Yazdani, and Karniadakis]{Raissi2020}
M.~Raissi, A.~Yazdani, and G.~E. Karniadakis.
\newblock Hidden fluid mechanics: Learning velocity and pressure fields from
  flow visualizations.
\newblock \emph{Science}, 367\penalty0 (6481):\penalty0 1026--1030, 2020.

\bibitem[Reissig(2003)]{Reissig2003}
M.~Reissig.
\newblock Hyperbolic equations with non-{L}ipschitz coefficients.
\newblock \emph{Rend. Semin. Mat. Univ. Politec. Torino}, 61\penalty0
  (2):\penalty0 135--181, 2003.

\bibitem[Rodriguez-Torrado et~al.(2022)Rodriguez-Torrado, Ruiz,
  Cueto-Felgueroso, Green, Friesen, Matringe, and
  Togelius]{RodriguezTorrado2022}
R.~Rodriguez-Torrado, P.~Ruiz, L.~Cueto-Felgueroso, M.~C. Green, T.~Friesen,
  S.~Matringe, and J.~Togelius.
\newblock Physics-informed attention-based neural network for hyperbolic
  partial differential equations: Application to the {Buckley--Leverett}
  problem.
\newblock \emph{Sci. Rep.}, 12:\penalty0 7557, 2022.

\bibitem[Rokhlin et~al.(2009)Rokhlin, Szlam, and Tygert]{Rokhlin2009}
V.~Rokhlin, A.~Szlam, and M.~Tygert.
\newblock A randomized algorithm for principal component analysis.
\newblock \emph{{SIAM} J. Matrix Anal. Appl.}, 31\penalty0 (3):\penalty0
  1100--1124, 2009.

\bibitem[Rudy et~al.(2017)Rudy, Brunton, Proctor, and Kutz]{Rudy2017}
S.~H. Rudy, S.~L. Brunton, J.~L. Proctor, and J.~N. Kutz.
\newblock Data-driven discovery of partial differential equations.
\newblock \emph{Sci. Adv.}, 3\penalty0 (4):\penalty0 e1602614, 2017.

\bibitem[Sch{\"a}fer and Owhadi(2024)]{Schafer2023}
F.~Sch{\"a}fer and H.~Owhadi.
\newblock Sparse recovery of elliptic solvers from matrix-vector products.
\newblock \emph{SIAM J. Sci. Comput.}, 46\penalty0 (2):\penalty0 A998--A1025,
  2024.

\bibitem[Sch{\"a}fer et~al.(2021)Sch{\"a}fer, Sullivan, and
  Owhadi]{Schafer2021}
F.~Sch{\"a}fer, T.~J. Sullivan, and H.~Owhadi.
\newblock Compression, inversion, and approximate {PCA} of dense kernel
  matrices at near-linear computational complexity.
\newblock \emph{Multiscale Model. Simul.}, 19:\penalty0 688--730, 2021.

\bibitem[Shi and Townsend(2021)]{Shi2021}
T.~Shi and A.~Townsend.
\newblock On the compressibility of tensors.
\newblock \emph{{SIAM} J. Matrix Anal. Appl.}, 42\penalty0 (1):\penalty0
  275--298, 2021.

\bibitem[Smith(1998)]{Smith1998}
H.~F. Smith.
\newblock A parametrix construction for wave equations with {$C^{1,1}$}
  coefficients.
\newblock \emph{Ann. Inst. Fourier}, 48\penalty0 (3):\penalty0 797--835, 1998.

\bibitem[Stewart and Sun(1990)]{Stewart1990}
G.~W. Stewart and J.-G. Sun.
\newblock \emph{Matrix Perturbation Theory}.
\newblock Academic Press, 1990.

\bibitem[Subramanian et~al.(2023)Subramanian, Harrington, Keutzer, Bhimji,
  Morozov, Mahoney, and Gholami]{Subramanian2023}
S.~Subramanian, P.~Harrington, K.~Keutzer, W.~Bhimji, D.~Morozov, M.~Mahoney,
  and A.~Gholami.
\newblock Towards foundation models for scientific machine learning:
  Characterizing scaling and transfer behavior.
\newblock In \emph{Advances in Neural Information Processing Systems},
  volume~37, pages 71242--71262, 2023.

\bibitem[Thodi et~al.(2022)Thodi, Ambadipudi, and Jabari]{Thodi2023}
B.~T. Thodi, S.~V.~R. Ambadipudi, and S.~E. Jabari.
\newblock Learning-based solutions to nonlinear hyperbolic {PDEs}: Empirical
  insights on generalization errors.
\newblock In \emph{Advances in Neural Information Processing Systems Workshop
  on Machine Learning and the Physical Sciences}, 2022.

\bibitem[Townsend(2014)]{Townsend2014}
A.~Townsend.
\newblock \emph{Computing with functions in two dimensions}.
\newblock PhD thesis, University of Oxford, 2014.

\bibitem[Townsend and Trefethen(2015)]{Townsend2015}
A.~Townsend and L.~N. Trefethen.
\newblock Continuous analogues of matrix factorizations.
\newblock \emph{Proc. R. Soc. Lond. A}, 471\penalty0 (2173):\penalty0 20140585,
  2015.

\bibitem[Trefethen(2010)]{Trefethen2010}
L.~N. Trefethen.
\newblock Householder triangularization of a quasimatrix.
\newblock \emph{IMA J. Numer. Anal.}, 30\penalty0 (4):\penalty0 887--897, 2010.

\bibitem[Trefethen(2019)]{Trefethen2019}
L.~N. Trefethen.
\newblock \emph{Approximation Theory and Approximation Practice}.
\newblock SIAM, extended edition, 2019.

\bibitem[Vershynin(2018)]{Vershynin2018}
R.~Vershynin.
\newblock \emph{High-Dimensional Probability: An Introduction with Applications
  in Data Science}.
\newblock Cambridge University Press, 2018.

\bibitem[Wan et~al.(2023)Wan, Zepeda-N{\'u}{\~n}ez, Boral, and Sha]{Wan2023}
Z.~Y. Wan, L.~Zepeda-N{\'u}{\~n}ez, A.~Boral, and F.~Sha.
\newblock Evolve smoothly, fit consistently: Learning smooth latent dynamics
  for advection-dominated systems.
\newblock In \emph{International Conference on Learning Representations}, 2023.

\bibitem[Wang(2023)]{Wang2023}
H.~Wang.
\newblock New error bounds for {L}egendre approximations of differentiable
  functions.
\newblock \emph{J. Fourier Anal. Appl.}, 29\penalty0 (42), 2023.

\bibitem[Wang et~al.(2021)Wang, Wang, and Perdikaris]{WangS2021}
S.~Wang, H.~Wang, and P.~Perdikaris.
\newblock Learning the solution operator of parametric partial differential
  equations with physics-informed {DeepONets}.
\newblock \emph{Sci. Adv.}, 7:\penalty0 eabi8605, 2021.

\bibitem[Zepeda-N{\'u}{\~n}ez and Demanet(2016)]{ZepedaNunez2016}
L.~Zepeda-N{\'u}{\~n}ez and L.~Demanet.
\newblock The method of polarized traces for the {2D} {H}elmholtz equation.
\newblock \emph{J. Comput. Phys.}, 308:\penalty0 347--388, 2016.

\bibitem[Zhang and Lin(2018)]{Zhang2018}
S.~Zhang and G.~Lin.
\newblock Robust data-driven recovery of governing physical laws with error
  bars.
\newblock \emph{Proc. R. Soc. Lond. A}, 474\penalty0 (2217):\penalty0 20180305,
  2018.

\end{thebibliography}

\end{document}